\documentclass[10pt]{article}

\makeatletter
\def\th@plain{%
  \thm@notefont{}
  \itshape 
}
\def\th@definition{%
  \thm@notefont{}
  \normalfont 
}
\makeatother

\usepackage{amsthm}
\usepackage{mathrsfs}
\usepackage{amsmath}
\usepackage{amsfonts}
\usepackage{latexsym}
\usepackage{amssymb}
\usepackage[pdftex]{graphicx}
\usepackage{srcltx,graphicx,bm}
\usepackage{multirow}
\usepackage{soul}
\usepackage{tabularx,ragged2e}
\newcolumntype{C}{>{\Centering\arraybackslash}X} 
\usepackage{caption}
\usepackage{subcaption}
\usepackage{float}
\usepackage{xcolor}
\usepackage{pdfpages}
\usepackage[a4paper]{geometry}
\usepackage{color}
\usepackage{hyperref}
\usepackage[outdir=./]{epstopdf}
\usepackage{verbatim}

\newtheorem{theorem}{Theorem}[section]
\newtheorem{assumption}[theorem]{Assumption}
\newtheorem{corollary}[theorem]{Corollary}
\newtheorem{definition}[theorem]{Definition}
\newtheorem{notation}[theorem]{Notation}
\newtheorem{example}[theorem]{Example}
\newtheorem{lemma}[theorem]{Lemma}

\newtheorem{proposition}[theorem]{Proposition}
\newtheorem{remark}[theorem]{Remark}

\numberwithin{equation}{section}

\newcommand{\imp}{\mathrm{imp}}
\newcommand{\wv}{\widehat{v}}
\newcommand{\tv}{\widetilde{v}}
\newcommand{\bw}{\boldsymbol{w}}
  \newcommand{\bcTh}{\boldsymbol{\mathcal{T}}\!_h\!}

\newcommand{\bzero}{\backslash\{0\}}
\newcommand{\beq}{\begin{equation}}
\newcommand{\eeq}{\end{equation}}

\newcommand{\Ccont}{C_{\mathrm{cont}}}

\newcommand{\cA}{{\mathcal A}}

\newcommand{\cI}{{\mathcal I}}
\newcommand{\cL}{{\mathcal L}}
\newcommand{\cP}{{\mathcal P}}

\newcommand{\bcT}{\boldsymbol{\mathcal{T}}}
\newcommand{\bcL}{\boldsymbol{\mathcal{L}}}
\newcommand{\bcI}{\boldsymbol{\mathcal{I}}}

\newcommand{\bcU}{\boldsymbol{\mathcal{U}}}
\newcommand{\cR}{\mathcal{R}}  
\newcommand{\tcR}{\widetilde{\mathcal{R}}}

\newcommand{\CR}{\mathcal{R}}

\newcommand{\cT}{{\mathcal T}}

\newcommand{\cU}{{\mathcal U}}

\newcommand{\cB}{{\mathcal B}}

\newcommand{\bx}{\boldsymbol{x}}

\newcommand{\be}{\boldsymbol{e}}
\newcommand{\bv}{\boldsymbol{v}}
\newcommand{\bg}{\boldsymbol{g}}
\newcommand{\bGamma}{\boldsymbol{\Gamma}}

\newcommand{\supp}{\mathrm{supp}}

\newcommand{\ta}{\widetilde{a}}

\newcommand{\Rea}{\mathbb{R}}

\newcommand{\eps}{\varepsilon}

\newcommand{\ri}{{\rm i}}

\definecolor{myblue}{rgb}{0,0,0.6}
\definecolor{darkgreen}{rgb}{0,0.5,0}
\hypersetup{colorlinks=true,
linkcolor=myblue,citecolor=myblue,filecolor=myblue,urlcolor=myblue}

\definecolor{escol}{rgb}{0,0,0.8}
\definecolor{sgcol}{rgb}{0,0,0.7}
\definecolor{estcol}{rgb}{0.5,0,0}
\definecolor{esnewcol}{rgb}{0,0.5,0}

\newcommand{\beqs}{\begin{equation*}}
\newcommand{\eeqs}{\end{equation*}}
\newcommand{\bit}{\begin{itemize}}
\newcommand{\eit}{\end{itemize}}
\newcommand{\ben}{\begin{enumerate}}
\newcommand{\een}{\end{enumerate}}
\newcommand{\bal}{\begin{align}}
\newcommand{\eal}{\end{align}}
\newcommand{\bals}{\begin{align*}}
\newcommand{\eals}{\end{align*}}
\newcommand{\bre}{\begin{remark}}
\newcommand{\ere}{\end{remark}}
\newcommand{\bpf}{\begin{proof}}
\newcommand{\epf}{\end{proof}}
\newcommand{\ble}{\begin{lemma}}
\newcommand{\ele}{\end{lemma}}
\newcommand{\bco}{\begin{corollary}}
\newcommand{\eco}{\end{corollary}}
\newcommand{\bex}{\begin{example}}
\newcommand{\eex}{\end{example}}
\newcommand{\bth}{\begin{theorem}}
\newcommand{\enth}{\end{theorem}}

\newcommand{\dist}{{\rm dist}}
\newcommand{\tfa}{\text{  for all  }}

\newcommand{\ton}{\text{ on }}

\newcommand{\tand}{\text{ and }}

\newcommand{\boldnu}{\boldsymbol{\nu}}

\newcommand*{\N}[1]{\left\|#1\right\|}

\newcommand{\pdiff}[2]{\frac{\partial #1}{\partial #2}}

\newcommand{\boldmu}{\boldsymbol{\mu}}

\title{Convergence of Restricted Additive Schwarz with impedance transmission conditions for discretised Helmholtz problems} 

\author{Shihua Gong, Ivan G.~Graham and Euan A.~Spence 
  \\[2ex]
  {\tt  sg2328@bath.ac.uk, I.G.Graham@bath.ac.uk, E.A.Spence@bath.ac.uk}, \\[2ex] Department of Mathematical Sciences, University of Bath, Bath BA2
7AY, UK.}
\date{\today}

\begin{document}
\maketitle

\begin{abstract}
The  Restricted Additive  Schwarz method  with impedance  transmission
conditions, also known as the  Optimised Restricted Additive Schwarz (ORAS)  method,  is a simple overlapping  one-level parallel domain 
decomposition method, which has been successfully used as an iterative solver and as a preconditioner for   discretized Helmholtz boundary-value problems.
In this paper, we give, for the first time, a convergence analysis for ORAS as an iterative solver -- and also as a preconditioner -- for nodal finite element Helmholtz systems of any polynomial order. The analysis starts by showing  (for general domain decompositions)  that ORAS is  an unconventional finite element approximation of a classical parallel iterative Schwarz method, formulated at the PDE (non-discrete) level. This non-discrete Schwarz method was recently analysed in [Gong, Gander, Graham, Lafontaine, Spence, arXiv 2106.05218], and the present paper gives a corresponding discrete version of this analysis. In particular, for  domain decompositions in strips in 2-d, we show that, when the mesh size is small enough, 
ORAS inherits the convergence properties of the Schwarz method, 
independent of polynomial order. The proof relies on characterising the ORAS iteration 
in terms of discrete `impedance-to-impedance maps', which we prove (via a novel weighted finite-element error analysis)  converge as $h\rightarrow 0$ in the operator norm to their non-discrete counterparts.
 \end{abstract}

\section{Introduction}

The Helmholtz equation, as the time-harmonic form of the wave equation, arises in many scientific and engineering applications,  including acoustics, seismic imaging  and earthquake modelling. Solutions at high-frequency are often required; for example, in inverse scattering  when  imaging  fine details of a  scatterer. However, solving the discretised Helmholtz equation at high frequency  is very challenging, first because of the huge number of degrees of freedom (required, in general, to approximate the oscillating solutions) and second because the system matrices are
  non-Hermitian  and highly  indefinite. There exist considerable efforts in the literature for finding efficient
algorithms for solving discrete Helmholtz problems, e.g., to name two of the most successful, the ``shifted Laplace'' preconditioner
using a multigrid strategy  \cite{ErOoVu:06} and  sweeping algorithms and their variants, e.g., \cite{EnYi:11c,Chen13a, TaZeHeDe:19},  which can be viewed as
multiplicative domain decomposition methods.
Substantial recent reviews of solvers for discrete Helmholtz problems can  be found, for example in \cite{GaZh:19} and in the introductions to \cite{GrSpZo:20} and \cite{TaZeHeDe:19}.  However, most practical methods are justified by empirical experiments and there are relatively few rigorous convergence analyses for usable methods.

In this paper, we give a new analysis of the Restricted Additive Schwarz preconditioner with local impedance transmission conditions, also called the ORAS (`Optimised Restricted Additive Schwarz') method, which is arguably the most successful one-level inherently-parallel domain decomposition method for Helmholtz problems.  
ORAS can be applied on very general geometries, does not require parameter-tuning, and 
(as the numerical experiments in \cite{GoGaGrLaSp:21} show)
can even be robust to increasing frequency (in the sense that, in some situations, the number of iterations is independent of frequency).
More generally, ORAS can be combined with coarse spaces  to improve its robustness properties; see, e.g., \cite{BoDoJoTo:20}.  ORAS has been used as a stand-alone solver and also in conjunction with a coarse correction (sometimes applied hierarchically) in substantial scientific applications (e.g., \cite{MEDIMAX,BoDoJoTo:20,BoDoGrSpTo:17c}).
The ORAS preconditioner has quite a large literature, e.g., \cite{StGaTh:07,DoJoNa:15, KiSa:07} but,
  although there are  arguments partially explaining its success   (e.g., \cite{Efstathiou}, \cite[\S2.3.2]{DoJoNa:15}),
  there is no rigorous convergence theory for Helmholtz problems. It is worth mentioning that there has been considerable recent interest in convergence theory for non-overlapping domain decomposition methods for Helmholtz problems, e.g., \cite{MoRoAnGe:20,ClPa:20,DeNiTh:20}; these algorithms and the corresponding analyses are  quite distinct from the method and analysis given here. Here we are motivated to investigate overlapping methods because of the wide practical use of  (variants of) Restricted Additive Schwarz methods (e.g., \cite{MEDIMAX}), because of the lack of theoretical understanding of these methods, and also  because our previous work showed promising performance of this method in practice (e.g., \cite{BoDoGrSpTo:17c,GoGrSp:21}). Moreover in the recent paper \cite{GoGaGrLaSp:21} the benefit of overlap was shown rigorously at the non-discrete (PDE) level.    

Since the system matrices  arising from Helmholtz problems are  typically non-Hermitian and indefinite, the Krylov method of choice is GMRES,  and a standard   approach for predicting its convergence   uses the `Elman estimate',
  which requires that the norm of the system matrix is bounded above  and  its field of values 
   is bounded away from the origin.  This theory was successfully applied to a simple variant  of the ORAS
  preconditioner (called SORAS), for Helmholtz problems with some absorption \cite{GrSpZo:20,GoGrSp:21}.
  However this theory was unable to predict a convergence rate for the ORAS (or SORAS)  preconditioner  applied to the Helmholtz problem with no absorption; in fact the field of values of the ORAS preconditioned system matrix,  for a particular example,  was shown in \cite[Figure 2]{GoGrSp:21} to include the origin of the complex plane and to have a growing  boundary as the frequency increases. Nevertheless, \cite{GoGrSp:21} also shows that  ORAS still works well as preconditioner for the Helmholtz problem with no absorption and here  we develop a novel analysis (avoiding analysing the field of values) that explains this.  
  
\subsection{The Helmholtz problem and its discretisation}
 Although the ORAS method can be applied to very general scattering problems in general geometries,  we give its analysis here for the model interior impedance problem defined as follows.   Let $\Omega\subset \Rea^d$, $d=2,3$, be a bounded Lipschitz domain.   
Given  the source function $f \in L^2(\Omega)$, the impedance data  $g \in L^2(\partial \Omega)$,   and the frequency $k>0$, we consider the problem
of finding the solution $u$  of 
\begin{equation} \label{eq:Helm}
\begin{aligned}
 \Delta u + k^2 u & =-f  \quad \text{in} \quad \Omega,\\
 \partial_n u -\ri k u &=g \quad\ton \partial \Omega,
  \end{aligned}   
\end{equation}
where $\partial_n$ denotes the outward normal derivative on $\partial \Omega$. 
  We write this problem in  variational form: find $ u \in H^1(\Omega)$ such that 
\beqs
a(u,v)= F(v) \tfa v \in H^1(\Omega), 
\eeqs
where, for all $u,v \in H^1(\Omega)$, 
\beq
a(u,v) := \int_\Omega \nabla u \cdot\overline{\nabla v} - k^2 u\,\overline{v} - \int_{\partial \Omega} \ri k u\,\overline{v}
\quad\tand\quad F(v) :=\int_\Omega f \, \overline{v} + \int_{\partial \Omega} g \,\overline{v}. \label{eq:defF}
\eeq

We are concerned with solvers for the discrete version of   \eqref{eq:Helm} in   
a  nodal finite element space $V \subset H^1(\Omega)$ consisting of continuous piecewise  polynomials of total degree  $\leq p$ on a conforming simplicial mesh $T^h$ with mesh diameter $h$, 
 i.e., we seek the solution $u_h\in V$ to the problem 
\beq\label{eq:Helm-weak-discrete}
a(u_h,v_h)= F(v_h) \tfa v_h \in V.
\eeq
This paper concerns  iterative solvers and preconditioners for computing $u_h$, especially for $k$ large. Let $V'$ be the dual space of $V$.
It is useful to define the discrete
 operators $\cA_h :V \mapsto V'$ and $F_h\in V'$
 by
\begin{equation}
  (\cA_h u_h)(v_h) : = a(u_h,v_h)\quad \text{and}   \quad   F_h(v_h) =  F(v_h) \quad \text{for all} \quad u_h , v_h \in V_h ,\label{eq:discrete}  
\end{equation}
and then to write equation \eqref{eq:Helm-weak-discrete}, equivalently, as the equation 
\begin{align} \cA_h u_h = F_h\label{eq:abstract} \end{align}
to be solved in $V'$ for the solution $u_h \in V$. Under a mesh resolution condition linking $h$ and $k$, this equation always has a unique solution $u_h$ (see Theorem \ref{th:DuWu} below). 
    
\subsection{Related previous work at the PDE level}
\label{subsec:Related} 

In the recent paper \cite{GoGaGrLaSp:21} we studied  the particular
  parallel overlapping Schwarz method (specified in \eqref{eq21}-\eqref{star} below) for  the problem \eqref{eq:Helm} at the PDE (i.e., non-discrete) level.
In \eqref{eq21}-\eqref{star},    $\{\Omega_j\}_{j = 1}^N$, are a set of
    $N$
 polygonal or Lipschitz  polyhedral    
subdomains    
forming  an overlapping cover of the global domain $\Omega$. (When we come to the finite element counterpart, we assume that each subdomain boundary $\partial \Omega_j$ is resolved by the finite element mesh $T^h$.)
  The method \eqref{eq21}-\eqref{star} can be thought of as a generalisation of the classical  algorithm of
  Despr\'{e}s \cite{De:91},  \cite{BeDe:97} to the case of overlapping subdomains.


The convergence theory for  \eqref{eq21}-\eqref{star}  given in    \cite{GoGaGrLaSp:21} starts  by showing that 
the vector of errors on each subdomain:
   \begin{align} \label{31} \be^n = (e_1^n, e_2^n, \ldots e_N^n)^\top , \quad \text{where} \quad  e_j^n =  u\vert_{\Omega_j}  - u_j^n,  \quad j = 1, \ldots, N,  \end{align}
(where $u$ is the solution of \eqref{eq:Helm} and   $u_j^n$ is given in \eqref{eq21}-- \eqref{eq23})  satisfies  a fixed-point iteration of the form
 \begin{align}\label{eq:errit} 
     \be^{n+1} = \bcT \be^{n},   
 \end{align}
 with  $\bcT$  defined by
 \eqref{eq41}-\eqref{eq43}.
The paper \cite{GoGaGrLaSp:21} showed that, for   general Lipschitz subdomains, the fixed-point operator $\bcT$ acts in the tensor product  of Helmholtz-harmonic functions (i.e., solutions of $(\Delta +k^2)u=0$) on each subdomain. Using this setting, \cite{GoGaGrLaSp:21} characterised the powers of $\bcT$ in terms of (compositions of)
 ``impedance-to-impedance maps'' linking pairs of subdomains with non-trivial intersection. 

In addition, the paper \cite{GoGaGrLaSp:21} then gave, for the special case of 2-d rectangular domains  covered by overlapping strips,  sufficient conditions for 
    $\bcT^{sN}$ to be a contraction, where $N$ is the number of subdomains and $s$ is a small integer. These conditions were  formulated in terms of norms of (compositions of) impedance-to-impedance maps; the conditions were investigated (and verified in certain cases) by both analytical and numerical techniques.
To illustrate the theory, experiments with finite element approximations of the parallel Schwarz method were given in \cite{GoGaGrLaSp:21}.  
 
\subsection{The main results of this paper} 

While the results in \cite{GoGaGrLaSp:21} 
gave sufficient conditions for the power contractivity of  the Schwarz method at the PDE level for domain decompositions in strips, they did not prove the power contractivity of the corresponding method in the practical finite element case (namely, ORAS, defined explicitly in \S \ref{subsec:ORAS}). The present paper  establishes, for the first time, conditions for ORAS to be power contractive.  The main results of the present  paper (and the technical obstacles that had to be overcome) are as follows. 

  \bigskip
  
  \noindent   $\bullet$ \  For  general Lipschitz domains and  subdomains (in 2- and 3-d), Theorem \ref{prop:eqORAS} shows that the ORAS iteration
  can be interpreted as a  finite element approximation of the parallel Schwarz method  \eqref{eq21}-\eqref{star}; a subtlety of this interpretation is an appropriate discretization of the boundary condition \eqref{eq22}.\\

\noindent   $\bullet$  \     In this general set-up, ORAS is formulated in  \eqref{eq:errit_disc} as a fixed point iteration of the form 
  $$
  \be_h^{n+1} = \bcTh\, \be_h^n,  
  $$
  where $ \be_h^n : = (e_{h,1}, \ldots, e_{h,N})^\top$ is the
  discrete analogue of \eqref{31} and is defined precisely in \S \ref{subsec:prop}.
It is shown in \S \ref{subsec:discrete} that 
    $e_{h,\ell}^n$
  is  discrete Helmholtz-harmonic 
  on  $\Omega_\ell$ (i.e., is a solution  at the finite-element level of the homogeneous Helmholtz problem on  $\Omega_\ell$).
The impedance data of a
  discrete Helmholtz-harmonic function on $\Omega_\ell$ is introduced in Lemma \ref{lm:iso} and this defines a norm    (subject to a  mesh resolution condition) on the space of discrete Helmholtz-harmonic functions (see Proposition \ref{lm:norm}); this norm is the  discrete analogue of the boundary impedance norm  (or `pseudo-energy') introduced in  
  \cite{De:91} and is used to analyse the power contractivity of $\bcTh$.\\
    
  \noindent
  $\bullet$\ 
   Theorem  \ref{thm:main_T2} shows that, for  decompositions in strips in 2-d,
 the action of (powers of) $\bcTh$ can be expressed in terms of
  (compositions of) discrete impedance-to-impedance maps,  analogous to the theory at the continuous level in \cite[\S4]{GoGaGrLaSp:21}.
\\

\noindent
$\bullet$\ Theorem \ref{thm:main} proves that, for these 2-d strip decompositions, the discrete impedance-to-impedance maps converge in the $L^2$ norm to their continuous counterparts as $h \rightarrow 0$ for any fixed $k$; this result shows that the numerical procedure used in \cite{GoGaGrLaSp:21} 
 for computing the norms of the impedance-to-impedance maps is reliable.
To prove norm convergence of the discrete impedance-to-impedance maps,  we adapt the elliptic-projection argument introduced in \cite{FeWu:09,FeWu:11} for Helmholtz finite-element analysis to a weighted setting,
using enhanced  regularity at interior interfaces.
\\

   \noindent
  $\bullet$\ The main result of the paper, Theorem \ref{thm:Tnnorm},  is that, for these 2-d strip decompositions,
if the parallel Schwarz method is power contractive, then so is ORAS for $h$ sufficiently small; 
 furthermore, the power contractivity of ORAS is independent of the polynomial degree of the finite elements  (see Corollary \ref{cor:Tnnorm}). 
 The proof that  ORAS ``inherits'' the properties of the parallel Schwarz method is not obvious, because the discrete and continuous fixed-point operators, $\bcTh$ and $\bcT$, operate in different spaces (here called $\mathbb{V}_0$ and $\mathbb{U}_0$); nevertheless we show that, for all $n$,   $\Vert \bcT_h^n\Vert_{\mathbb{V}_0} \rightarrow \Vert \bcT^n\Vert_{\mathbb{U}_0}$ as  $h \rightarrow 0$.  \\
  
 \noindent
  $\bullet$\ Since the experiments in \cite{GoGaGrLaSp:21} essentially used the ORAS method to illustrate the parallel Schwarz method, we just summarise these in \S \ref{sec:numerical}, but we also  provide some additional experiments,  especially designed to illustrate 
 that the number of ORAS iterations is independent of $h$ and $p$.

 \bigskip
 
\noindent {\bf Organisation of paper.} \   \S \ref{sec:ORAS} introduces the ORAS and the parallel Schwarz methods and explores their connections.
\S \ref{sec:ORAS-conv} obtains the fixed point iteration satisfied by ORAS and examines the spaces in which it operates. \S \ref{sec:strip} restricts attention to decompositions of a rectangular domain into  strips, introduces the discrete impedance-to-impedance maps, and proves that,  for $h$ small enough,     ORAS is power contractive if the parallel Schwarz method has
the same property. \S\ref{subsec:power_recap} recalls results from \cite{GoGaGrLaSp:21} that give conditions for the power contractivity of the parallel Schwarz method. \S\ref{sec:imp} shows the norm convergence of the discrete impedance-to-impedance maps.  \S \ref{sec:numerical} gives numerical experiments.

\subsection{Wellposedness of discrete Helmholtz problems}

The Helmholtz problem \eqref{eq:Helm} at the continuous level is well-posed for all $k>0$ (see, e.g., \cite[\S8]{SaBrHa:19}).
The following result guarantees the existence of a unique solution $u_h \in V$ to the discrete problem \eqref{eq:abstract}. 

\begin{theorem}\label{th:DuWu} If $\Omega$ is convex, then there exists a dimensionless constant $C$, independent of $h$, $k$, and the diameter of $\Omega$ (but dependent on $p$) such that if 
\begin{align}\label{eq:defhk} 
h^{2p}k^{2p+1} \leq \frac{C}{\rm diam(\Omega)}, 
\end{align} 
then the problem \eqref{eq:abstract} has a unique solution $u_h \in V$. \end{theorem}

\bpf[References for the proof]
This result, without the explicit dependence on ${\rm diam}(\Omega)$,  was proved in \cite[Corollary 6.11]{DuWu:15}.
The result with this explicit dependence (and for the more-general case when the Helmholtz equation has variable coefficients) was proved in \cite[Theorem 2.39]{Pe:20}; indeed, the condition \eqref{eq:defhk} is \cite[equation (2.62)]{Pe:20}, the constants $\mathscr{C}(n)$ and $C_{\rm cond}$ are dimensionless (see their definitions in (2.65) and Table 2.4, respectively), and the constant $C_{\rm stab}$ defined in (2.61) can be written as a non-dimensional quantity multiplied by ${\rm diam}(\Omega)$ (see, e.g., \cite[Equation 3.5]{LaSpWu:19a} and the associated discussion).
\epf

In the rest of the paper we implicitly  assume that $u_h$ exists, 
so that the discussion of iterative methods for finding $u_h$ makes sense. Note that: 
(i) estimates for the error  $u-u_h$ can also be proven under the condition \eqref{eq:defhk} (see \cite{DuWu:15} and, e.g., the overviews in
\cite[Remark 2.9]{GrSpZo:20}, \cite[\S1]{LaSpWu:19a}), and
(ii) if $\Omega$ is not convex, then well-posedness of \eqref{eq:abstract} can still be proved, but under 
more restrictive conditions on $h$ and $k$ (since one no longer has $H^2$ regularity of the solution; see the discussion in the last paragraph of Remark \ref{rem:more_general_geometries}).
  
 \section{Restricted Additive Schwarz with impedance transmission condition (ORAS) }
 \label{sec:ORAS}
\subsection{Definition of the ORAS method} 
\label{subsec:ORAS}
We assume that each subdomain boundary $\partial \Omega_j$ is resolved by the finite element mesh $T^h$.
We denote the  diameter of $\Omega_j$ by  $H_j$ and set $H = \max_j H_j$. We introduce  a partition of unity  $\{\chi_j\}_{j = 1}^N$, with the properties 
 \beq
 \left.
 \begin{array}{ll} 
  \text{for each} \quad j, & \  \supp \chi_j \subset \overline{\Omega_j}, \quad
   0 \leq  \chi_j(\bx) \leq 1\,\,  \text{when } \bx \in \overline{\Omega_j},\\
   & \\
   \quad\tand\quad  & 
  \sum_j  \chi_j(\bx) = 1 \,\tfa \bx \in \overline{\Omega}.
 \end{array}
 \right\} \label{POUstar}
 \eeq
 \begin{notation}\label{not:sesqui} 
  On any curve (in 2-d)/surface (in 3-d) $\Gamma \subset \Omega$, we
     let $\langle v, w \rangle_\Gamma = \int_\Gamma v \overline{w} $.  On  any subdomain $D\subset \Omega$,  we introduce the local sesquilinear forms 
\begin{align*}
 \ta_{D}(v ,w)  := \int_{D} \nabla v \cdot\overline{\nabla w} - k^2 v\,\overline{w} \quad \text{and} \quad a_D(v,w) := \ta_D(v,w) - \ri k \langle v,w\rangle_{\partial D};
\end{align*}
note that $a_D$  arises in  the variational formulation of a local Helmholtz impedance problem on $D$. When $D=\Omega$, we denote $\ta_\Omega(\cdot,\cdot)$ by $\ta(\cdot,\cdot)$ and $a_\Omega(\cdot,\cdot)$ by $a(\cdot,\cdot)$.
\end{notation}
Define the local finite element space on $\overline{D}$ by  
 $V(D): =\{v_{h}|_{\overline{D}}~:~v_{h}\in {V}\}$ with corresponding nodes  $N^h(\overline{D})   \subset N^h(\overline{\Omega})$, where $N^h(\overline{\Omega})$ denotes the nodes of $V$.  For simplicity, we write $V_j$ instead of  $V(\Omega_j)$. 
Then, analogously
to \eqref{eq:discrete} we  
 define $\cA_{h,j}: V_j \rightarrow V_{j}'$ \ by
$$(\cA_{h,j}v_{h,j})(w_{h,j}) : = a_{\Omega_j}(v_{h,j}, w_{h,j}) \quad \tfa v_{h,j} , w_{h,j} \in V_j. $$

We assume that $h$ is small enough so that each local problem is well-posed:
\begin{assumption} \label{ass:localDuWu} 
For each $j = 1,\ldots, N$, 
$\cA_{h,j}:V_j \rightarrow V_{j}'$ is invertible.  
\end{assumption}

\noindent {\bf Remark. }  If $\Omega$ and each $\Omega_j, j=1,\ldots,N$ are  convex, and \eqref{eq:defhk} holds, then Assumption \ref{ass:localDuWu} holds,  since 
$(\mathrm{diam}(\Omega))^{-1} \leq H_j^{-1}$.\\

We also need prolongations and restrictions linking local and global problems. For any subset $D$ of $\overline{\Omega}$ (which we assume to be a union of elements of the mesh $T^h$), we define the prolongation $\cR_{h,D} ^\top: V(D) \rightarrow V$ defined for all $v_{h,D} \in V(D)$ by
\begin{equation}\label{nodewise}
(\cR_{h,D} ^\top  v_{h,D})(x_i)  = \left\{ \begin{array}{ll} v_{h,D} (x_i) & \quad {x_i \in N^h(\overline{D})}, \\ 
                                          0 & \quad x_i \in N^h(\overline{\Omega})\backslash N^h(\overline{D}). \end{array} \right.   
\end{equation}
Note that the extension  $\cR_{h,D}^\top v_{h,D}$ is defined {\em nodewise}: It  coincides with $v_{h,D}$ at  finite element nodes in  $\overline{D}$ and vanishes at nodes in $\overline{\Omega}\backslash \overline{D}$. Thus $\cR_{h,D} ^\top  v_{h,D} \in V \subset H^1(\Omega)$  is a finite element approximation of the zero extension of $v_{h,D}$ to all of ${\Omega}$.  (The zero extension itself --  \eqref{eq:zero-ext} below --  is  not in $H^1(\Omega)$ in general.)
For simplicity, we denote $\cR_{h,\Omega_j}^\top: V_j \rightarrow V$ by $\cR_{h,j}^\top $.

Then we can  define the restriction operator $\cR_{h,j}: V' \rightarrow V_{ j}'$  by duality, i.e., 
  \begin{align} \label{eq:duality1} (\cR_{h,j}G_h)(v_{h,j}) := G_{h}(\cR_{h,j}^\top v_{h,j}) \quad \tfa \ v_{h, j} \in V_j, \ G_h \in V', .\end{align}

Finally,  in restricted additive Schwarz methods, prolongation from local to global is done using a partition of unity, and so
 we also define a weighted prolongation $\tcR_{h,j} ^\top: V_j \rightarrow V$ by, for all $v_{h,j}\in V_j,$
$$
\tcR_{h,j}^\top v_{h,j}  = \cR_{h,j}^\top (\Pi_h (\chi_j v_{h,j})),
$$
where $\Pi_h$ is the nodal interpolation onto $V$.  
Note that
\begin{align} \label{eq:prores} \sum_\ell \tcR^\top_{h, \ell} (w_h\vert_{\Omega_\ell}) \ = \ w_h,  \tfa  w_h \in V. \end{align} 

                                      Then the ORAS preconditioner is  the operator $\cB_h^{-1}  : V' \rightarrow V$ defined by
\begin{align} \label{eq:PC} \cB_h^{-1} \ : =\  \sum_j \tcR_{h,j}^\top \cA_{h,j}^{-1}\cR_{h,j}. 
\end{align}
The corresponding preconditioned Richardson iterative method for \eqref{eq:abstract}  can be written as
     \begin{equation}\label{eq:oras-iter}
     u^{n+1}_h = u^n_h  + \cB_h^{-1}  (F_h - \cA_h u_h^n);
   \end{equation}
   the matrix version of this is discussed in \cite{GoGaGrSp:21}. 
In \S\ref{sec:2.3}, we identify \eqref{eq:oras-iter} as a discrete version of the following
    classical iterative method for the Helmholtz problem at the PDE level.   
    
 \subsection{The Schwarz method with impedance transmission condition}
 \label{subsec:parallelSchwarz}
Starting with the Helmholtz problem \eqref{eq:Helm} and the domain decomposition $\{\Omega_j\}_{j=1}^N$,
the parallel Schwarz method is the following:~given the $n$th iterate $u^n$ defined on $\Omega$, solve the local problem on $\Omega_j$ for  $u^{n+1}_j$ ,
\begin{align}
  (\Delta + k^2)u_j^{n+1}  & = - f  \quad  &\text{in } \  \Omega_j,  \label{eq21}\\
  \left(\frac{\partial }{\partial n_j} - \ri k \right) u_j^{n+1}  & = \left(\frac{\partial }{\partial n_j} - \ri k \right) u^n \quad  &\text{on }  \Gamma_j:= \partial \Omega_j\backslash \partial \Omega, \label{eq22}  \\
    \left(\frac{\partial }{\partial n_j} - \ri k \right) u_j^{n+1}  & = g  \quad  &\text{on } \  \partial \Omega_j  \cap \partial \Omega,  \label{eq23} 
\end{align}
where $\partial /\partial n_j$ denotes the outward normal derivative on $\partial \Omega_j$.
The new iterate $u^{n+1}$ is then the weighted sum of the local solutions
\begin{align} \label{star}
  u^{n+1} := \sum_j \chi_j  u_j ^{n+1}.
\end{align}
This method is  a generalization of the classical  algorithm of Despr\'{e}s \cite{De:91},  \cite{BeDe:97} to the case of overlapping subdomains, and  \cite{GoGaGrLaSp:21} performs a  convergence analysis of it. 
To derive  the problem satisfied by the error, note that if $u$ denotes the  solution   of \eqref{eq:Helm},   then,   $u_j : = u\vert_{\Omega_j}$
satisfies 
\begin{align}
    (\Delta + k^2)u_j & = - f  \quad  &\text{in } \  \Omega_j, \label{eq11} \\
  \left(\frac{\partial }{\partial n_j} - \ri k \right) u_j & = \left(\frac{\partial }{\partial n_j} - \ri k \right) u \quad  &\text{on }  \ \Gamma_j,  \label{eq12} \\
  \left(\frac{\partial }{\partial n_j} - \ri k \right) u_j & = g  \quad  &\text{on } \  \partial \Omega_j  \cap \partial \Omega,\label{eq13}
\end{align}
 Then, using    \eqref{star} and \eqref{POUstar},  
 the global error $e^n := u - u^n $  can be written as
\begin{align*}
e^n =\sum_\ell \chi_\ell   u\vert_{\Omega_\ell }  -\sum_\ell    \chi_\ell u_\ell^n = \sum_\ell \chi_\ell  e_\ell^n , \quad \text{where} \quad e_\ell^n = u|_{\Omega_\ell} - u_\ell^n. 
\end{align*}
Thus, subtracting  \eqref{eq21}-\eqref{eq23} from  \eqref{eq11}-\eqref{eq13}, we obtain 
\begin{align}
  (\Delta + k^2)e_j^{n+1}  & = 0  \quad  \text{in } \  \Omega_j,  \label{eq31}\\
  \left(\frac{\partial }{\partial n_j} - \ri k \right) e_j^{n+1}  & = \left(\frac{\partial }{\partial n_j} - \ri k \right) e^n  \nonumber \\
    & = \ \sum_\ell \left(\frac{\partial }{\partial n_j} - \ri k \right) \chi_\ell e_\ell^n  \quad  \text{on }  \ \Gamma_j\label{eq32}  \\
    \left(\frac{\partial }{\partial n_j} - \ri k \right) e_j^{n+1}  & = 0  \quad  \text{on } \  \partial \Omega_j  \cap \partial \Omega.   \label{eq33} 
\end{align} 
In \cite{GoGaGrLaSp:21}, this method was analysed in the following spaces.
For $D$ a Lipschitz domain, let
 $$U(D) := \big\{ u \in H^1(D): \,  \Delta u + k^2 u \in L^2(D), \ \partial u / \partial n - \ri k u \in L^2(
   \partial D) \big\},   $$
and   let the corresponding `Helmholtz harmonic' space be defined by
 \begin{align} \label{eq:defUzero}  U_0 (D) := \big\{ u \in H^1(D): \,  \Delta u + k^2 u = 0
   \ \text{in} \  D, \ \partial u / \partial n - \ri k u \in L^2(
   \partial D) \big\} \subset U(D);  \end{align}
  we equip $U_0 (D)$  with the `boundary (pseudo-)energy norm'  (see \cite{De:91}, \cite[\S3.1]{GoGaGrLaSp:21})
\begin{align}\label{eq:pseudo-energy}
\|v\|_{U_0(D)}^2 \ := \  
\N{\pdiff{v}{n}-\ri k v }^2_{L^2(\partial D)}. 
   \end{align}
   The error vector \eqref{31}  thus satisfies the fixed-point iteration \eqref{eq:errit}, where the matrix of operators  $\bcT = (\cT_{j,\ell})_{j,\ell = 1}^N$ is  defined as follows.  
For any $j,\ell \in \{1,\ldots,N\}$ (not necessarily equal), we define, for $v_\ell \in U(\Omega_\ell)$, 
\begin{align}
  (\Delta + k^2)(\cT_{j,\ell} v_\ell)  & = 0  \quad  \text{in } \  \Omega_j,  \label{eq41}\\
  \left(\frac{\partial }{\partial n_j} - \ri k \right) (\cT_{j,\ell} v_\ell) & = \left(\frac{\partial }{\partial n_j} - \ri k \right) (\chi_\ell v_\ell) \quad  \text{on }  \ \Gamma_j, \label{eq42}  \\
    \left(\frac{\partial }{\partial n_j} - \ri k \right) ({\cT_{j,\ell} v_\ell}) & = 0  \quad  \text{on } \  \partial \Omega_j \cap \partial \Omega.   \label{eq43} 
\end{align}
 The analysis in \cite{GoGaGrLaSp:21} then shows that $\bcT$ is a bounded operator on the tensor product space 
 \begin{align} \label{eq:bU0}\mathbb{U}_0 :=  \prod_{j=1}^N U_0(\Omega_j),\end{align}  
Moreover,   under certain assumptions, \cite{GoGaGrLaSp:21} shows that $\bcT^N$ is a contraction, thus guaranteeing the convergence of the parallel Schwarz  method \eqref{eq21}-\eqref{star}. This paper proves analogous estimates for ORAS; we begin by connecting ORAS to 
\eqref{eq21}-\eqref{star}.

\subsection{Connection between  \eqref{eq21}-\eqref{star} and  ORAS}\label{sec:2.3}
In this subsection, we show that a finite element approximation of \eqref{eq21}-\eqref{star} yields \eqref{eq:oras-iter}.
 First note that the variational form of \eqref{eq21}-\eqref{eq23} is: find $u_j^{n+1}\in H^1(\Omega_j)$ such that
\begin{equation}\label{eq:local-cont}
a_{\Omega_j}(u_j^{n+1} ,v_j) = F(\cR_j^\top v_j) + \left\langle \left(\frac{\partial }{\partial n_j} - \ri k \right) u^n,  v_j \right\rangle_{\Gamma_j}\quad \text{for all }v_j\in H^1(\Omega_j),
\end{equation}
where $F$ as given by \eqref{eq:defF} and
  $\cR_j^\top $ is the 
zero extension operator with domain $H^1(\Omega_j)$:
\begin{equation}\label{eq:zero-ext}
\cR_j^\top v_j = \left\{ 
\begin{aligned}
v_j \quad &\text{in } \overline{\Omega_j}, \\
0 \quad &\text{in } \overline{\Omega}\backslash \overline{\Omega_j}.
\end{aligned}\right.
\end{equation}

It is not immediately clear how best to approximate the second term on the right-hand side of  \eqref{eq:local-cont} when $u^{n}$ is replaced by a finite element approximation.  
To understand how to deal with this issue, it is useful to introduce the following non-standard sesquilinear form associated with \eqref{eq:Helm}
\begin{align} \label{eq:alpha} \alpha(w,v) = - \int_\Omega ((\Delta + k^2)w) \overline{v} + \left\langle \left(\frac{\partial}{\partial n} - \ri k\right ) w , v\right\rangle_{\partial \Omega} ,
\end{align}
This is well-defined for   $w \in U(\Omega)$ and  $v \in L^2(\Omega)$, as long as $v$ has $L^2(\partial \Omega)$ trace. Indeed,  
by applying Green's identity, we see that  
\begin{align} \label{eq:consistency}
  \alpha(w,v) = a(w,v) \quad \text{if} \quad  w \in U(\Omega) \quad \text{and} \quad  v \in H^1(\Omega).
  \end{align} 
Another application of Green's identity gives the following formula for the boundary integral term in \eqref{eq:local-cont}.  
\begin{proposition} \label{prop0}
If $(w, v_j)\in U(\Omega) \times H^1(\Omega_j)$, then 
\begin{align}\label{eq:imp-cont}
\left\langle \left(\frac{\partial }{\partial n_j} - \ri k \right) w,  v_j \right\rangle_{\Gamma_j} &=a_{\Omega_j}(w, v_j)- \alpha(w, \cR_j^\top v_j)\\
&=:  b_j(w,v_j) \nonumber
\end{align}
\end{proposition}
Even though $\cR_j^\top v_j$ is not, in general, in $H^1(\Omega)$, the right-hand side of \eqref{eq:imp-cont} is well-defined.
To obtain our finite element analogue of  \eqref{eq:local-cont}, we  define an analogue of $b_j$ on the  product space $H^1(\Omega) \times V_j$ by 
\begin{align} \label{eq:defbhj}
\ b_{h,j}(w, v_{h,j}) := a_{\Omega_j}(w, v_{h,j}) - a(w, \cR_{h,j}^\top v_{h,j})\quad
\tfa(w , v_{h,j}) \in H^1(\Omega)  \times V_j.
 \end{align} 
To obtain \eqref{eq:defbhj} from \eqref{eq:imp-cont} we have replaced $\cR_j^\top$ (simple extension by zero) by the node-wise finite element extension operator $\cR_{h,j}^\top$. Since  $\cR_{h,j}^\top v_{h,j} \in H^1(\Omega)$ we can then use \eqref{eq:consistency} to replace $\alpha $ in \eqref{eq:imp-cont} by $a$ in \eqref{eq:defbhj}.   

We now show formally that $b_{h,j}$ in \eqref{eq:defbhj}     approximates the boundary integral $b_j$ in \eqref{eq:imp-cont}. (Later we prove a convergence result for this approximation in a specific geometric setting; see Definition \ref{def:imp-can} and Theorem \ref{thm:main}.) Let
 $v_{h,j} \in V_j \cap  H^1_{0,\Gamma_j}(\Omega_j)$, where  $\Gamma_j$ is defined in \eqref{eq22} and 
 \begin{align} \label{eq:zeroG} H_{0,\Gamma_j}^1(\Omega_j) = \{v\in H^1(\Omega_j)~:~ v=0 \text{ on }\Gamma_j\}.\end{align}
By \eqref{nodewise} and \eqref{eq:zero-ext},
$\cR^\top_{h,j} v_{h,j} = \cR^\top_j v_{h,j}$ and so $a(w,\cR^\top_{h,j} v_{h,j}) = a_{\Omega_j} (w, v_{h,j}) $ and thus 
\eqref{eq:defbhj} implies that
\begin{align}b_{h,j}(w, v_{h,j}) = 0 \quad \tfa \quad  v_{h,j}\in V_j \cap  H^1_{0,\Gamma_j}(\Omega_j);
  \label{eq:zero}\end{align} 
this is the discrete analogue of the fact that $b_j$ is a boundary integral over $\Gamma_j$. 

We use $b_{h,j}$ to obtain  the finite element analogue of \eqref{eq:local-cont},  which we  combine with the  analogue of \eqref{star} to obtain the following. 
\begin{definition}[Finite element analogue of parallel Schwarz method]
Given $u_h^n \in V$, let $u_{h,j}^{n+1} \in V_j$ be the solution of
\begin{equation}\label{eq:local-discrete}
a_{\Omega_j}(u^{n+1}_{h,j} ,v_{h,j}) = F(\cR_{h,j}^\top v_{h,j}) + b_{h,j}(u^n_h, v_{h,j}),  \quad \tfa  v_{h,j}\in V_j,
\end{equation}
and then set 
\begin{equation}\label{star-discrete}
u^{n+1}_{h} := \sum_\ell \tcR^\top_{h,\ell}u^{n+1}_{h,\ell},
\end{equation}
\end{definition}
The next theorem shows \eqref{eq:local-discrete}, \eqref{star-discrete} is equivalent to  ORAS. 
\begin{theorem}[Connection between  the parallel Schwarz method and  ORAS] 
\label{prop:eqORAS}
  With the same starting guess,  the iterates produced by \eqref{eq:local-discrete},  \eqref{star-discrete} coincide with those produced by    \eqref{eq:oras-iter}. 
\end{theorem}
\begin{proof}
  Combining \eqref{eq:local-discrete} with the definition  \eqref{eq:defbhj} of $b_{h,j}$, and using \eqref{eq:duality1} and \eqref{eq:discrete}, we obtain
  \begin{align} a_{\Omega_j}(u_{h,j}^{n+1} - u_h^n\vert_{\Omega_j} , v_{h,j} ) & = F_h(\cR_{h,j}^\top v_{h,j})
    - a(u_h^n, \cR_{h,j}^\top v_{h,j}) \label{eq:errja} \\
    &=  (\cR_{h,j} (F_h - \cA_h u_h^n))(v_{h,j}).
    \nonumber
    \end{align}
  Hence $ u_{h,j}^{n+1} - u_h^n\vert_{\Omega_j}  \ = \ \cA_{h,j}^{-1} \cR_{h,j} (F_h - \cA_h u_h^n)$ and,  by \eqref{star-discrete} and \eqref{eq:PC}, 
  \begin{align} u_h^{n+1} &= \sum_j \tcR^\top_{h,j} u_{h,j}^{n+1} \nonumber \\
  &= \sum_j \tcR^\top_{h,j} u_{h}^n\vert_{\Omega_j} + \cB_h^{-1} (F_h - \cA_h u_h^n). \label{eq:RHS}\end{align}
  By \eqref{eq:prores}, the first term on the right-hand side of \eqref{eq:RHS}  is  $u_{h}^n$, so this formula  coincides  with \eqref{eq:oras-iter}. 
\end{proof}

The following lemma  is used  frequently in the remainder of the analysis. 
\begin{lemma}\label{lm:Thlj}
(i) For all $(w , v_{h,j}) \in H^1(\Omega) \times V_j$,
\begin{align}
  b_{h,j}(w,v_{h,j}) & = - \ta_{\Omega \backslash \Omega_j} (w, \cR_{h,j}^\top v_{h,j})  + \ri k \langle w, \cR_{h,j}^\top v_{h,j}\rangle_{\partial \Omega\backslash \partial \Omega_j} - \ri k \langle w, v_{h,j}\rangle_{\Gamma_j} .
   \label{eq:bell}  
  \end{align}
  
  \noindent (ii) If either $\ell = j$ or $\Omega_\ell \cap \Omega_j = \emptyset$, then $$b_{h,j}(\tcR_{h,\ell}^\top w_{h,\ell},v_{h,j}) = 0   \quad \tfa \ w_{h,\ell} \in V_\ell \quad \text{and} \quad  v_{h,j} \in V_j. $$  
\end{lemma}

\begin{proof}
  To prove part (i) we use  the definition \eqref{eq:defbhj} and recall Notation \ref{not:sesqui}, 
    to obtain 
\begin{align*}
  b_{h,j} (w,v_{h,j}) :&= \ta_{\Omega_j}(w ,v_{h,j}) - \ta(w,  \cR^\top_{h,j} v_{h,j}) 
                          -\ri k \langle w, v_{h, j}\rangle_{\partial \Omega_j} + \ri k \langle w, \cR_{h,j}^\top v_{h,j} \rangle_{\partial \Omega}. 
\end{align*}
By \eqref{nodewise}, $\cR_{h,j}^\top v_{h,j} = v_{h,j}$ on $\overline{\Omega_j}$, and thus
\begin{align*}
  b_{h,j} (w,v_{h,j}) :&= - \ta_{\Omega\backslash \Omega_j}(w ,\cR_{h,j}^\top v_{h,j}) 
                          -\ri k \langle w, v_{h, j}\rangle_{\partial \Omega_j \backslash \partial \Omega } + \ri k \langle w, \cR_{h,j}^\top v_{h,j} \rangle_{\partial \Omega\backslash \partial \Omega_j}. 
\end{align*}
Part (i) then follows since  $\Gamma_j = {\partial \Omega_j \backslash \partial \Omega}$.
Part (ii) follows from Part (i) using the facts that (a) $\tcR_{h,\ell}^\top w_{h,\ell}$ vanishes outside  $\overline{\Omega_\ell}$ and on $\Gamma_\ell$, and (b) if $\Omega_\ell\cap\Omega_j= \emptyset$ then $\Omega_j$ and $\Omega_\ell$ are separated by at least one layer of elements.
\end{proof}

\section{General properties of the  ORAS iteration}\label{sec:ORAS-conv}
\subsection{The error propagation operator}
\label{subsec:prop}


We now obtain the discrete analogue of \eqref{eq31}-\eqref{eq33}. Recall that $u_h$ is the finite element approximation of \eqref{eq:Helm} (as  defined in \eqref{eq:Helm-weak-discrete}) and $\Omega_j$ is the overlapping cover of $\Omega$ introduced in \S \ref{subsec:Related}. With  $u_{h,j}^n$ and $u_h^n$  as defined in \eqref{eq:local-discrete}  and \eqref{star-discrete}, we introduce the  local and global errors,   defined respectively, by
\begin{equation*}
e_{h,j}^n := u_h|_{\Omega_j} - u_{h,j}^n   \quad \text{and}\quad  e_h^n := u_h - u_h^n .
\end{equation*}

\begin{proposition}[The error equation]  \label{prop:error}
For each $j = 1,\ldots, N$, 
\begin{equation}\label{eq:error-form}
a_{\Omega_j}(e_{h,j}^{n+1} , v_{h,j}) = b_{h,j}(e_h^n, v_{h,j})
\end{equation} and 
  \begin{equation}\label{eq:errorPOU}
e_h^n = \sum_\ell \tcR_{h,\ell}^\top e_{h,\ell}^n.
\end{equation}

\end{proposition}

\begin{proof} To obtain  \eqref{eq:errorPOU}, we first observe that
  $$\sum_\ell \tcR_{h,\ell}^\top e_{h,\ell}^n = \sum_\ell \tcR_{h,\ell}^\top u_{h}\vert_{\Omega_\ell} - \sum_\ell \tcR_{h,\ell}^\top u_{h,\ell}^n = u_h - u_h^n , $$
where we used \eqref{eq:prores} and \eqref{star-discrete}. We then combine  \eqref{eq:errja} and  \eqref{eq:Helm-weak-discrete} to  obtain  
\begin{align}\label{eq:prop311}
a_{\Omega_j}(u^{n+1}_{h,j} - u^n_h\vert_{\Omega_j},v_{h,j}) =  a( u_h - u^n_h, \cR_{h,j}^\top v_{h,j}) = a(e_h^n, \cR_{h,j} ^\top v_{h,j} ). 
\end{align}
Since $ u^{n+1}_{h,j} - u^n_h\vert_{\Omega_j} = (u_h - u_h^n)\vert_{\Omega_j} - (u_h\vert_{\Omega_j} - u_{h,j} ^{n+1} ) = e^n_h\vert_{\Omega_j} - e^{n+1}_{h,j} $, substituting this in \eqref{eq:prop311} and rearranging yields  \eqref{eq:error-form}.
\end{proof}

The last result  motivates  the following definition.  

\begin{definition} [Discrete error propagation operator]\label{def:Tjl}
  For   $j=1, \ldots, N$, and  $w_{h,\ell} \in V_\ell$,
let $\cT_{h;j,\ell} w_{h,\ell} \in V_j$ be defined as the solution of
\begin{equation}\label{eq:def-Tjl}
a_{\Omega_j}( \cT_{h;j,\ell} w_{h,\ell} , v_{h,j}  )= b_{h,j}(\tcR_{h,\ell}^\top w_{h,\ell}, v_{h,j})\quad \tfa v_{h,j}\in V_j.
\end{equation}
\end{definition}

This definition and Proposition \ref{prop:error} imply that the error  of the ORAS iterative method satisfies the fixed point iteration 
\begin{align}\label{eq:errit_disc}
e_{h,j}^{n+1} =\sum_\ell \cT_{h;j,\ell} e_{h,\ell}^{n}  \quad \text{for} \ j = 1, \ldots , N,  \quad   \text{written compactly as} \quad  \be_h^{n+1} =\bcTh~ \be_h^{n},
\end{align}
where  $\be_h^n := (e_{h,1}^n,e_{h,2}^n,\cdots, e_{h,N}^n )^\top. $ In this setting   ${\bcT}_h$ is a operator acting on the tensor product space  $\mathbb{V} := \Pi_{\ell=1}^N V_\ell$.

We see immediately that \eqref{eq:errit_disc} is  the discrete analogue of the corresponding fixed point iteration for the parallel Schwarz method given in \eqref{eq:errit}. The convergence analysis for the Schwarz method in \cite{GoGaGrLaSp:21} is carried out in the space $\mathbb{U}_0$ (defined by \eqref{eq:bU0}). We now create a finite-element analogue of this space;
to do this, we need to address the fact that functions   $v \in U_0(\Omega_j)$ satisfy the homogeneous Helmholtz equation $\Delta v + k^2 v = 0$,
  but the Helmholtz operator is not defined on the finite element space 
$V_j$.  

 \subsection{Discrete Helmholtz-harmonic spaces}
\label{subsec:discrete}Recall that $\Gamma_\ell$ is defined by \eqref{eq22} and $H_{0,\Gamma_\ell}^1(\Omega_\ell)$ is defined by \eqref{eq:zeroG}.

\begin{definition}[Discrete Helmholtz-harmonic spaces]\label{def:3.3}
Let $V_{0,\ell}\subset V_\ell$ be defined by
\begin{align} \label{eq:discHH}
V_{0,\ell} \ : =\Big\{w_{h,\ell}\in V_\ell~:~ a_{\Omega_\ell}(w_{h,\ell}, v_{h,\ell}) = 0 \quad \tfa \ v_{h,\ell} \in V_\ell\cap H_{0,\Gamma_\ell}^1(\Omega_\ell)\Big\},
\end{align}
and let the corresponding tensor product space be defined by
$
\displaystyle{\mathbb{V}_0 \ : = \ \prod_{\ell=1}^N V_{0,\ell}. }
$
 \end{definition}
 
Observe that $V_{0,\ell}$ consists of finite-element approximations to functions in $U_0(\Omega_\ell)$ (defined in \eqref{eq:defUzero}) which  satisfy 
 $(\Delta +k^2)w=0$ in $\Omega_\ell$, 
 $\partial_n w - \ri k w =0$ on 
 $\partial\Omega_\ell\setminus \Gamma_\ell = \partial \Omega_\ell \cap \partial \Omega$.  
Note that in $V_{0,\ell}$ there is  no constraint on $\Gamma_\ell$ because of the zero Dirichlet boundary condition on $\Gamma_\ell$ in the space of test functions in \eqref{eq:discHH}. 
Recalling the equations \eqref{eq31}-\eqref{eq33} satisfied by the error $e_\ell^n$ at the PDE level, we see that $V_{0,\ell}$ is
   the discrete analogue  of the space in which the error $e_\ell^n$ of the parallel Schwarz method lies.
 To define a discrete analogue of the corresponding norm \eqref{eq:pseudo-energy}, we first prove the following lemma, which extracts the ``discrete impedance data'' of a function $w_{h,\ell} \in V_{0,\ell}$. 
 \begin{lemma}[Discrete impedance data] \label{lm:iso}
Given   $w_{h,\ell}\in V_{0,\ell}$, there exists a unique finite element function 
    $\imp^h_{\Gamma_\ell}w_{h,\ell}  \in V(\Gamma_\ell)$ (where $V(\Gamma_\ell)$ is the space of finite element functions on $\Gamma_\ell$) 
 such that
\begin{equation}\label{eq:defimpdata}
\big\langle \imp^h_{\Gamma_\ell} w_{h,\ell},  v_{h,\ell}\big\rangle_{\Gamma_\ell}  \ = \   a_{\Omega_\ell}(w_{h,\ell}, v_{h,\ell})    \quad \tfa \  v_{h,\ell}\in V_\ell.
\end{equation}
We refer to $\imp^h_{\Gamma_\ell} w_{h,\ell}$ as the   {\em discrete impedance data}  of $w_{h,\ell}$.   
\end{lemma}
\begin{proof} 
Let    $g_{h,\Gamma_{\ell}} \in V(\Gamma_\ell)$ be the unique solution of 
$$
\langle g_{h,\Gamma_\ell},  v_{h,\Gamma_\ell}\rangle_{\Gamma_\ell} \ = \ a_{\Omega_\ell}(w_{h,\ell}, \widehat{v}_{h,\ell}) \quad\tfa \  v_{h,\Gamma_\ell}\in V(\Gamma_\ell),
$$
where $\widehat{v}_{h,\ell}$ denotes the zero nodal
extension of $ v_{h,\Gamma_\ell}$ to $V_\ell$ (i.e., the finite element function that coincides with $ v_{h,\Gamma_\ell}$ on $\Gamma_\ell$
and has value zero at all other nodes of $\overline{\Omega_\ell}$). 

Now, given $v_{h,\ell} \in V_\ell$, let $ v_{h,\Gamma_\ell}$ be its
restriction to  $\Gamma_\ell$ and define  $\wv_{h,\ell}$ as above. Then   
$\tv_{h,\ell} :=  v_{h,\ell} - \widehat{v}_{h,\ell} \in V_\ell\cap H_{0,\Gamma_\ell}^1(\Omega_\ell)$. 
By \eqref{eq:discHH}, 
 $a_{\Omega_\ell}(w_{h,\ell}, \tv_{h,\ell}) = 0,$
and hence 
$$
a_{\Omega_\ell}(w_{h,\ell}, v_{h,\ell}) = a_{\Omega_\ell}(w_{h,\ell}, \widehat{v}_{h,\ell}) = \langle g_{h,\Gamma_\ell},  v_{h,\Gamma_\ell}\rangle_{\Gamma_\ell}  = \langle g_{h,\Gamma_\ell},  v_{h,\ell}\rangle_{\Gamma_\ell} ,
$$
which, with $\imp^h_{\Gamma_\ell} := g_{h,\ell}$,   is exactly \eqref{eq:defimpdata}. 
\end{proof}

Observe that if $w_{h,\ell} $ satisfies \eqref{eq:defimpdata}, then it is the finite element approximation of the problem $(\Delta + k^2)w = 0$ on $\partial \Omega_\ell$, with $w$ satisfying an impedance condition on $\partial \Omega_\ell$, with impedance data $\imp_{\Gamma_\ell}^h w_{h,\ell}$ on $\Gamma_\ell$  and $0$ on $\partial \Omega \backslash \Gamma_\ell$.    This fact justifies calling    $\imp_{\Gamma_\ell}^h w_{h,\ell}$  ``the discrete impedance data'' of $w_{h,\ell}$
and allows us to introduce the following norm on $V_{0,\ell}$, which is the analogue of
the norm on $U_0(\Omega_\ell)$    defined in \eqref{eq:pseudo-energy}. 
\begin{definition}[Discrete  Helmholtz-harmonic norm]  \label{def:bdnorm}
   For $w_{h,\ell} \in V_{0,\ell}$ 
       and \\ $\bw_h := (w_{h,1}, \ldots, w_{h,N})^\top \in \mathbb{V}_0$, let
\begin{equation}\label{eq:defimpnorm}
 \| w_{h,\ell}\|_{V_{0,\ell}}  :=  \|\imp^h_{\Gamma_\ell} w_{h,\ell}\|_{L^2(\Gamma_\ell)}
\quad \text{and} \quad \|\bw_{h}\|_{\mathbb{V}_0} := \Big(\sum_\ell\|w_{h,\ell}\|_{V_{0,\ell}}^2 \Big)^{1/2}.   \end{equation}
\end{definition}

\begin{proposition}\label{lm:norm}
  The expressions in \eqref{eq:defimpnorm} are indeed norms  on $V_{0,\ell}$ and $\mathbb{V}_0$.  
\end{proposition}
\begin{proof}
  Suppose $\Vert w_{h,\ell} \Vert_{V_{0,\ell}} = 0$ for some   $w_{h,\ell} \in V_{0,\ell}$.
  Then $\imp^h_{\Gamma_\ell} w_{h,\ell}$ must vanish on $\Gamma_\ell$. This implies, via
  \eqref{eq:defimpdata},  that
  $a_{\Omega_\ell}(w_{h,\ell}, v_{h, \ell} ) = 0 $ for all $v_{h,\ell} \in V_\ell$.  
  Then, by Assumption \ref{ass:localDuWu},  $w_{h, \ell} = 0$.
  The other norm axioms follow easily from the definitions in  \eqref{eq:defimpnorm}. \end{proof}

\begin{proposition}\label{prop:Tdefd}   
  For each $j,\ell$, 
  the operator $\cT_{h;j, \ell}$ given in \eqref{eq:def-Tjl} is well-defined and maps $V_\ell$ to $V_{0,j}$. Moreover
$\bcTh: \mathbb{V} \rightarrow \mathbb{V}_0$.
 \end{proposition}
 \begin{proof} By Assumption \ref{ass:localDuWu}, the problem \eqref{eq:def-Tjl} has a unique solution and so $\cT_{h,j,\ell}$ is well-defined. Moreover, combining \eqref{eq:zero} with \eqref{eq:def-Tjl} and \eqref{eq:discHH}, we see that  $ \cT_{h;j,\ell}v_{h,\ell} \in V_{0,j}$ for all $v_{h,\ell} \in V_\ell$, as required. 
 \end{proof}

 \section{Convergence theory for strip domain decompositions }
 \label{sec:strip}
 
\subsection{The error propagation operator}

In this subsection we obtain a convergence theory for  \eqref{eq:local-discrete}-\eqref{star-discrete} when the domain is partitioned into strips, mirroring the analogous theory
in \cite[\S4]{GoGaGrLaSp:21} for \eqref{eq21}-\eqref{eq23}.
 We assume that   $\Omega$ is a rectangle of height $H$ and the subdomains $\Omega_\ell$ also have height $H$ and are bounded by vertical sides denoted by $\Gamma_\ell^-$ and $\Gamma_\ell^+$. We assume the subdomain $\Omega_\ell$ is only overlapped by $\Omega_{\ell-1}$ and $\Omega_{\ell+1}$ (with $\Omega_0 =\emptyset= \Omega_N$). Recalling that $\Gamma_\ell:= \partial\Omega_\ell\setminus \partial \Omega$, we have the following simple decomposition
$$\Gamma_\ell = \Gamma_\ell^+\cup \Gamma_\ell^-. $$
The geometry and notation is illustrated in Figure \ref{fig:overlap}.
\begin{figure}[H]
  \begin{center}
\scalebox{0.9}{
  \includegraphics{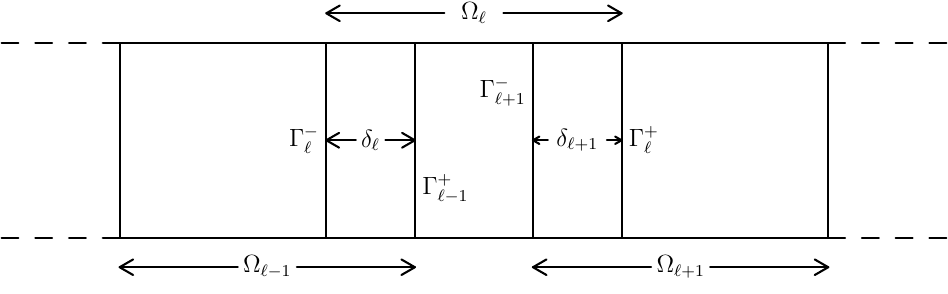}
}  
\end{center}   \caption{Three overlapping subdomains in 2-d \label{fig:overlap}}
\end{figure}
Using Part (ii) of Lemma \ref{lm:Thlj} in the definition 
of $\cT_{h:j,\ell}$
(Definition \ref{def:Tjl}), we see that  $\bcTh$ takes the  tridiagonal  form    
\begin{equation}\label{eq:splitT}
\bcTh =\begin{pmatrix}
0		&	\cT_{h;1,2} \\
\cT_{h;2,1}	&	0		&\cT_{h;2,3}\\
		&	\cT_{h;3,2}	&0		&\cT_{h;3,4}\\
		&			&	\ddots	&	\ddots&	\ddots\\
		&			&	\cT_{h;N-1,N-2}&		0	&\cT_{h;N-1,N}\\
		&			&			&	\cT_{h;N,N-1}&0	
\end{pmatrix}  =: \bcL_h + \bcU_h\, , 
\end{equation}
where $\bcL_h$ and $\bcU_h$ are the lower and upper triangular components of $\bcT_h$.

\subsection{Decomposition of $V_{0,\ell}$ and discrete impedance data }

The following subspaces of $V_{0,\ell}$ and $\mathbb{V}_{0}$ characterise the range spaces of the operators $\bcL_h, \bcU_h$  in \eqref{eq:splitT}.

\begin{definition}[Discrete Helmholtz-harmonic spaces with impedance data on $\Gamma_\ell^{\pm}$] For  $s \in \{+,-\}$, let
\begin{align} \label{eq:discHHs}
V_{0,\ell}^s :=\Big\{w_{h,\ell}\in V_\ell~:~ a_{\Omega_\ell}(w_{h,\ell}, v_{h,\ell}) = 0 \quad \tfa \ v_{h,\ell} \in V_\ell\cap H_{0,\Gamma_\ell^s}^1(\Omega_\ell)\Big\},
\end{align}
and let the corresponding tensor-product space be defined by
$$
\mathbb{V}^s_{0} =\big\{ \bw_h \in \mathbb{V}~:~ w_{h,\ell}\in V_{0,\ell}^s \quad \tfa  \ \ell = 1, \ldots, N \big\}. 
$$
\end{definition}
By Definition \ref{def:3.3} and Lemma \ref{lm:iso},
functions in $w_{h,\ell} \in V_{0,\ell}^{\pm}$ satisfy the finite element approximation of the homogeneous interior impedance problem on $\Omega_\ell$ with impedance data supported only on $\Gamma_\ell^{\pm}$. In the next lemma we give a simple formula for the norm on $V_{0,\ell}^s$.

\begin{lemma}\label{lm:imp-part}
  Suppose  $w_{h,\ell}\in V_{0,\ell}^s$ and $s\in \{+,-\}$.
Then ${\rm imp}_{\Gamma_\ell}^h w_{h,\ell} = 0$ on $\Gamma_\ell\backslash \Gamma_\ell^s$ and  
   \begin{align} \Vert w_{h,\ell} \Vert_{V_{0,\ell}}  = \Vert \imp_{\Gamma_\ell^s}^h w_{h,\ell}\Vert_{L^2(\Gamma_\ell^s)} \quad \text{where}  \quad \imp_{\Gamma_\ell^s}^h w_{h,\ell} := (\imp^h_{\Gamma_{\ell}} w_{h,\ell})\vert_{\Gamma_\ell^s}. 
\label{eq:42A}\end{align}
\end{lemma}

\begin{proof} Let $w_{h,\ell} \in V_{0,\ell}^s\subset V_{0,\ell}$. 
  By Lemma \ref{lm:iso}, $w_{h,\ell}$ satisfies \eqref{eq:defimpdata} for all $v_{h,\ell} \in V_\ell$. Now choose $v_{h,\ell}$ to be the zero nodal extension of the function
  $(\imp_{\Gamma_\ell}^h w_{h,\ell})\vert_{\Gamma_\ell \backslash \Gamma_\ell^s}$.
  Then $v_{h,\ell} \in V_\ell\cap H^1_{0,\Gamma_\ell^s}$ and,  by \eqref{eq:discHHs},
  $a_{\Omega_\ell}(w_{h,\ell}, v_{h,\ell}) = 0 $. Thus, the definition \eqref{eq:defimpdata} of $\imp_{\Gamma_\ell}^h w_{h,\ell}$ implies that $\Vert \imp_{\Gamma_\ell}^h w_{h,\ell} \Vert_{L^2(\Gamma_\ell\backslash \Gamma_\ell^s)}^2 = 0$,  and  the result follows.   \end{proof}

Using this result we can now obtain the following decomposition of $V_{0,\ell}$, which is crucial in the convergence theory in \S \ref{subsec:power}. 
\begin{lemma}[Decomposition of $V_{0,\ell}$]  \label{prop:decomp} {For each} $\ell = 1,\ldots, N$, 
  \begin{align}\label{eq:dec}
  \ V_{0,\ell} = V_{0,\ell}^- \oplus V_{0,\ell}^+. 
  \end{align}
  Moreover, if $w_{h,\ell} = w_{h,\ell}^- + w_{h,\ell}^+$ with $w_{h,\ell}^s \in V_{0,\ell}^s$, then
  \begin{align} \label{eq:normdec}
    \Vert w_{h,\ell}\Vert_{V_{0,\ell}}^2 \ =\  \Vert w_{h,\ell}^-\Vert_{V_{0,\ell}^-}^2 + \Vert w_{h,\ell}^+\Vert_{V_{0,\ell}^+}^2. \end{align}
 \end{lemma}

 \begin{proof} 
  Let $w_{h,\ell} \in V_{0,\ell} $; for each $s \in \{ +, -\}$, define $\imp^h_{\Gamma_\ell^s} w_{h,\ell} $ as in the second equation in \eqref{eq:42A}   
     and then  define $w_{h,\ell}^s \in V_{0,\ell}$ by  
     \begin{align*}  a_{\Omega_\ell} (w_{h,\ell}^s , v_{h,\ell}) = \big\langle \imp^h_{\Gamma_\ell^s} w_{h,\ell} , v_{h,\ell}\big\rangle_{\Gamma_\ell^s}   \quad \tfa \ v_{h,\ell} \in V_\ell. 
     \end{align*}
Then, for all $v_{h,\ell} \in V_\ell$, 
       \begin{align*}
        a_{\Omega_\ell}(w_{h,\ell}^- + w_{h,\ell}^+ , v_{h,\ell})
       &= \big\langle \imp^h_{\Gamma_\ell^-}w_{h,\ell} , v_{h,\ell}\big\rangle_{\Gamma_\ell^-}
       + \big\langle \imp^h_{\Gamma_\ell^+}w_{h,\ell} , v_{h,\ell}\big\rangle_{\Gamma_\ell^+}
        = \big\langle  \imp^h_{\Gamma_\ell}w_{h,\ell}  , v_{h,\ell}\big\rangle_{\Gamma_\ell}.
           \end{align*}
         So, by  \eqref{eq:defimpdata} and Assumption \ref{ass:localDuWu},
         $w_{h,\ell} = w_{h,\ell}^-  + w_{h,\ell}^+$.

We now show that the sum in \eqref{eq:dec} is direct. Suppose  $w_{h,\ell} \in V_{0,\ell}^- \cap V_{0,\ell}^+$.
         Then we can  write any $v_{h,\ell}\in V_\ell$  as $v_{h,\ell} =  v_{h,\ell}^- + v_{h,\ell}^+$ where $v_{h,\ell}^s \in V_\ell\cap H_{0,\Gamma_\ell^{s}}^1(\Omega_\ell)$, and then, by \eqref{eq:discHHs},
         $$
         a_{\Omega_\ell}(w_{h,\ell}, v_{h,\ell}) = a_{\Omega_\ell}(w_{h,\ell}, v_{h,\ell}^-) + a_{\Omega_\ell}(w_{h,\ell}, v_{h,\ell}^+) = 0; $$ 
         therefore $w_{h,\ell} = 0$ by Assumption \ref{ass:localDuWu}. The expression \eqref{eq:normdec} then follows since, by \eqref{eq:42A},
         \begin{align*}
         \|w_{h,\ell}\|_{V_{0,\ell}}^2 = \|  \imp^h_{\Gamma_\ell}w_{h,\ell} \|_{L^2(\Gamma_\ell )}^2
         &= \|  \imp^h_{\Gamma_\ell^+}w_{h,\ell} \|_{L^2(\Gamma_\ell ^+)}^2 + \|  \imp^h_{\Gamma_\ell^-}w_{h,\ell} \|_{L^2(\Gamma_\ell ^-)}^2 \\
         &=  \Vert w_{h,\ell}^-\Vert_{V_{0,\ell}^-}^2 + \Vert w_{h,\ell}^+\Vert_{V_{0,\ell}^+}^2.
        \end{align*}
        \end{proof}

\begin{lemma}[Mapping properties of $\cT_{h,j,\ell}$]  \label{lm:Tmaps} 
\begin{align} \label{eq:381} (i) \quad \cT_{h;j, j-1}: V_{j-1} \rightarrow V_{0,j}^-  \quad \text{and}  \quad
  (ii) \quad \cT_{h;j, j +1}: V_{j+1} \rightarrow V_{0,j}^+.\end{align}  Moreover 
$\bcL_h: \mathbb{V} \rightarrow \mathbb{V}_{0}^-$ and $\bcU_h: \mathbb{V} \rightarrow \mathbb{V}_{0}^+$.
 \end{lemma}
 \begin{proof} 
     We give the proof of \eqref{eq:381} (i) only; the proof of (ii) is analogous, and the statement about $\bcL$ and $\bcU$ follows from \eqref{eq:381}. Let  $w_{h,j-1} \in V_{j-1}$ and  $v_{h,j}\in  V_j\cap {H_{0,\Gamma_j^-}^1(\Omega_j)}$.
     Then, by Lemma \ref{lm:Thlj}, 
     \begin{align*}
       b_{h,j}(\tcR_{h,j-1}^\top w_{h,j-1},v_{h,j})   =& - \ta_{\Omega \backslash \Omega_j} (\tcR_{h,j-1}^\top w_{h,j-1}, \cR_{h,j}^\top v_{h,j}) \nonumber  \\
&+ \ri k \big\langle \tcR_{h,j-1}^\top w_{h,j-1}, \cR_{h,j}^\top v_{h,j}\big\rangle_{\partial \Omega\backslash \partial \Omega_j} - \ri k \big\langle \tcR_{h,j-1}^\top w_{h,j-1}, v_{h,j}\big\rangle_{\Gamma_j} .
     \end{align*}
   Now     $\tcR_{h,j-1}^\top w_{h,j-1} $ is supported inside $\overline{\Omega_{j-1}}$ and so vanishes on
   $\Gamma_j^+$. Moreover,   $v_{h,j}$ is supported in $\overline{\Omega_j}$ and vanishes on  $\Gamma_{j}^-$. Thus       $\cR_{h,j}^\top v_{h,j}$ also  vanishes
     in $(\Omega\backslash \Omega_j)\cap \Omega_{j-1}$ and on $(\partial \Omega\backslash \partial \Omega_j)\cap \partial \Omega_{j-1}$. Thus 
      \begin{align} \label{eq:neededlater} b_{h,j}(\tcR_{h,j-1}^\top w_{h,j-1},v_{h,j}) = 0.
      \end{align} 
     By 
   \eqref{eq:def-Tjl},
  $
  a_{\Omega_j}(\cT_{h,j,j-1} w_{h,j -1}, v_{h,j}) =0  \tfa \ v_{h,j}\in  V_j\cap {H_{0,\Gamma_j^-}^1(\Omega_j)}
  $, and thus $\cT_{h,j,j-1} w_{h,j -1} \in V_{0,j}^-$ by \eqref{eq:discHHs}.
     \end{proof}

\subsection{Discrete impedance-to-impedance maps}
\label{subsec:imptoimp}

In \cite{GoGaGrLaSp:21}, power  contractivity of $\bcT$  
is proved by studying the so-called   `impedance-to-impedance maps'  defined as follows.
Given $g\in L^2(\Gamma_\ell^s)$, where  $s\in \{+,-\}$,  let  $w_\ell\in H^1(\Omega_\ell)$ be the solution of $$a_{\Omega_\ell}(w_\ell,v_\ell) =  \langle g, v_{\ell} \rangle_{\Gamma_{\ell}^s} \quad \tfa \  v_{\ell}  \in H^1(\Omega_\ell), $$
(i.e., $w_\ell$ solves the Helmholtz problem on $\Omega_\ell$ with vanishing source and impedance data $g$ on $\Gamma_\ell^s$ and zero elsewhere). Then   
 $\cI_{\Gamma_{\ell}^s \to \Gamma_{\ell-1}^+} $ and  $\cI_{\Gamma_{\ell}^s \to \Gamma_{\ell+1}^-}$  are the maps that take $g$ to the corresponding impedance data of  $w_\ell$ on $\Gamma_{\ell-1}^+$ and   $\Gamma_{\ell+1}^-$, respectively; i.e., 
  \begin{align}
  \cI_{\Gamma_{\ell}^s \to \Gamma_{\ell-1}^+} g = \left( \left(\frac{\partial}{\partial n_{\ell-1}} - \ri k\right)w_\ell\right) \Bigg|_{\Gamma_{\ell-1}^+}
  \quad \text{and}\quad 
    \cI_{\Gamma_{\ell}^s \to \Gamma_{\ell+1}^-} g = \left( \left(\frac{\partial}{\partial n_{\ell+1}} - \ri k\right)w_\ell\right)  \Bigg|_{\Gamma_{\ell+1}^-}.
  \label{eq:ctsimp} \end{align}

  To analyse the power contractivity of   $\bcTh$ , we need to study  the discrete versions of these impedance-to-impedance maps; we first introduce some notation. 
     \begin{notation}[Extension from and restriction to the boundary] \label{not:ext} Let $s \in \{+, -\}$. \\
       (i) If  $v_{h, \Gamma_\ell^s}$ denotes a generic  element of $V(\Gamma^s_\ell)$, then $\wv_{h,\ell}^s \in V_\ell$  denotes its zero nodal extension to $V_\ell$ (i.e., the finite element function that coincides with
       $v_{h,\Gamma_\ell^s}$  on $\Gamma_\ell^s$ and has value zero at all other nodes of $\overline{\Omega_\ell}$). \\
       (ii) Conversely, if $v_{h,\ell}$  denotes a generic element of $V_\ell$,  then  $v_{h,\Gamma_\ell^s}$ denotes
        its restriction to $\Gamma_\ell^s$ and, as above,   $\wv_{h,\ell}^s \in V_\ell$ denotes the 
       zero nodal extension  of $v_{h,\Gamma_\ell^s}$  to $\Omega_\ell$. We also define     $\tv_{h,\ell}^s = v_{h,\ell} - \wv_{h,\ell}^s \in V_\ell$, so that $\tv_{h,\ell}^s \in V_\ell $ vanishes on $\Gamma_\ell^s$.    
     \end{notation} 

We now define  discrete analogues of \eqref{eq:ctsimp}. 
We use the notation introduced in Notation \ref{not:ext} and also define
  \begin{align*}
  \Omega_{\ell,j} := \Omega_\ell \cap \Omega_j. 
  \end{align*}

\begin{definition}[Discrete impedance-to-impedance map] \label{def:impmap}
    For  $g \in L^2(\Gamma_{\ell}^s)$ with $s\in\{+,-\}$,  let $w_{h,\ell} \in 
    V_\ell^s$ be the unique solution of 
\begin{align}  \label{eq:imp-def1}
    a_{\Omega_\ell}( w_{h,\ell}, v_{h,\ell}) = \langle  g, {v_{h,\ell}} \rangle_{\Gamma_{\ell}^s} \quad \tfa \  v_{h,\ell} \in V_\ell.
\end{align}
Define 
$ \cI_{\Gamma_{\ell}^s \to \Gamma_{\ell-1}^+}^h: L^2(\Gamma_{\ell}^s)
     \rightarrow V(\Gamma_{\ell-1}^+)$
     and 
     $ \cI_{\Gamma_{\ell}^s \to \Gamma_{\ell+1}^-}^h: L^2(\Gamma_{\ell}^s)
     \rightarrow V(\Gamma_{\ell+1}^-)$ as the solutions of
 \begin{align}    \label{eq:imp-def2}
\big\langle  \cI_{\Gamma_{\ell}^s \to \Gamma_{\ell-1}^+}^h g, v_{h,\Gamma_{\ell-1}^+}\big\rangle_{\Gamma_{\ell-1}^+}  = a_{\Omega_{\ell,\ell-1}} (w_{h,\ell},  \wv_{h,\ell-1}^+) - a_{\Omega_\ell}(w_{h,\ell}, \cR_{h,\ell-1}^\top  \wv_{h,\ell-1}^+), 
 \end{align} 
for any $v_{h,\Gamma_{\ell-1}^+} \in V(\Gamma_{\ell-1}^+)$,    and  
     \begin{align}    \label{eq:imp-def3}
\big\langle  \cI_{\Gamma_{\ell}^s \to \Gamma_{\ell+1}^-}^h g, v_{h, \Gamma_{\ell+1}^-}\big\rangle_{\Gamma_{\ell+1}^-}  = a_{\Omega_{\ell,\ell+1}} (w_{h,\ell},  \wv_{h,\ell+1}^-) - a_{\Omega_\ell}(w_{h,\ell},  \cR_{h,\ell+1}^\top \wv_{h,\ell+1}^-), 
 \end{align}
 for any $v_{h, \Gamma_{\ell+1}^-} \in V(\Gamma_{\ell+1}^-)$.
  \end{definition}
  
Observe that \eqref{eq:imp-def2} is the
   analogue  of  \eqref{eq:defbhj} with $\Omega$ replaced by $\Omega_\ell$ and
   $\Omega_j$ replaced by $\Omega_{\ell,\ell-1}$, and hence \eqref{eq:imp-def2} and \eqref{eq:imp-def3} can be seen as discrete analogues of \eqref{eq:ctsimp}.  To estimate  impedance-to-impedance maps we always  use the $L^2$ norm on the domain and co-domain; for simplicity we  denote this  as  $\Vert \cdot \Vert$. 
In Section \ref{sec:imp}, we  prove the following result. 
  \begin{theorem}\label{thm:imp-map-cont}
    Let $s \in \{+,-\}$.  Then
  \begin{align}
   \|\cI_{\Gamma_{\ell}^s \to \Gamma_{\ell-1}^+} - \cI_{\Gamma_{\ell}^s \to \Gamma_{\ell-1}^+}^h \| \to 0 \quad \text{and} \quad
   \|\cI_{\Gamma_{\ell}^s \to \Gamma_{\ell+1}^-} - \cI_{\Gamma_{\ell}^s \to \Gamma_{\ell+1}^-}^h \| \to 0 \quad \text{as} \quad  h\to 0,\label{eq:impmap_conv}
  \end{align}
where the first limit holds for $\ell=2,\ldots,N$, and the second limit holds for $\ell=1,\ldots,N-1$.
\end{theorem}

\subsection{Characterisation of  $\cT_{h,j,\ell}$ using  discrete impedance-to-impedance maps}

The following lemma helps connect the impedance maps defined in \eqref{eq:imp-def2} and \eqref{eq:imp-def3} with the action of the operator $\cT_{h;j,\ell} $, 
and is used to prove Theorem \ref{thm:main_T2} (the main result of this section).
  \begin{lemma}
Let $s\in\{+,-\}$,   $g \in L^2(\Gamma_{\ell}^s)$,      and suppose  $w_{h,\ell}$ satisfies \eqref{eq:imp-def1}. Then 
  \begin{align}   
\big\langle  \cI_{\Gamma_{\ell}^s \to \Gamma_{\ell-1}^+}^h g, v_{h,\ell-1}\big\rangle_{\Gamma_{\ell-1}^+}  = b_{h,\ell-1} (\tcR_{{h,\ell}}^\top w_{h,\ell} , v_{h,\ell-1 })\quad \tfa \ v_{h,\ell-1}\in V_{\ell-1}, \label{H3}
 \end{align}  
 and
   \begin{align}   
\big\langle  \cI_{\Gamma_{\ell}^s \to \Gamma_{\ell+1}^-}^h g, v_{h,\ell+1}\big\rangle_{\Gamma_{\ell+1}^-}  = b_{h,\ell+1} (\tcR_{{h,\ell}}^\top w_{h,\ell} , v_{h,\ell+1 })\quad \tfa \ v_{h,\ell+1}\in V_{\ell+1}. \label{H4}
 \end{align}  
  \end{lemma}
  \begin{proof}
  We only prove \eqref{H3}; the proof of \eqref{H4} is similar.
   Let  $v_{h,\ell-1} \in V_{\ell-1}$. We now apply the decomposition in Notation \ref{not:ext}(ii) (on $\Omega_{\ell-1}$ with $s$ equal to $+$), to write   $v_{h,\ell-1} = \wv_{h,\ell-1}^+ + \tv_{h,\ell-1}^+$, where $\wv_{h,\ell-1}^+$ coincides with $v_{h,\ell-1}$ on $\Gamma_{\ell-1}^+$ and is zero at other nodes and  $\tv_{h,\ell-1}^+$ vanishes on $\Gamma_{\ell-1}^+$. Thus, 
  \begin{equation*}
  \big\langle  \cI_{\Gamma_{\ell}^s \to \Gamma_{\ell-1}^+ }^h g, \tv_{h,\ell-1}^+\big\rangle_{\Gamma_{\ell-1}^+}= 0 = b_{h,\ell-1} (\tcR_{{h,\ell}}^\top w_{h,\ell} , {\tv}_{h,\ell-1 }^+),  
\end{equation*}
where the second equality follows by an argument analogous to that used to obtain  \eqref{eq:neededlater}.

Therefore, to complete the proof of \eqref{H3}, we only need to  prove that
  \begin{align} \label{eq:L3143}   \big\langle  \cI_{\Gamma_{\ell}^s \to \Gamma_{\ell-1}^+ }^h g, {\wv}_{h,\ell-1}^+\big\rangle_{\Gamma_{\ell-1}^+} =  b_{h,\ell-1} (\tcR_{{h,\ell}}^\top w_{h,\ell} , {\wv}_{h,\ell-1 }^+) . \end{align}  
  By \eqref{eq:imp-def2}, Notation \ref{not:ext} (ii)   and Notation \ref{not:sesqui},
  \begin{align}
    \big\langle  \cI_{\Gamma_{\ell}^s \to \Gamma_{\ell-1}^+}^h g, \wv_{h,\ell-1}^+ \big\rangle_{\Gamma_{\ell-1}^+}
    &  = \big\langle  \cI_{\Gamma_{\ell}^s \to \Gamma_{\ell-1}^+}^h g, v_{h,\Gamma_{\ell-1}^+} \big\rangle_{\Gamma_{\ell-1}^+}\nonumber \\
    &  = a_{\Omega_{\ell,\ell-1}} (w_{h,\ell}, \wv_{h,\ell-1}^+) - a_{\Omega_\ell}(w_{h,\ell},  {\cR_{h,\ell-1}^\top}\wv_{h,\ell-1}^+)
 \nonumber \\
& = \ta_{\Omega_{\ell,\ell-1}} (w_{h,\ell}, \wv_{h,\ell-1}^+ ) - \ta_{\Omega_\ell}(w_{h,\ell},  {\cR_{h,\ell-1}^\top}\wv_{h,\ell-1}^+ ) \nonumber \\
&  \qquad -\ri k \langle w_{h,\ell} , \wv_{h,\ell-1}^+ \rangle_{\partial \Omega_{\ell,\ell-1}} +\ri k\langle w_{h,\ell} , {\cR_{h,\ell-1}^\top}\wv_{h,\ell-1}^+ \rangle_{\partial \Omega_{\ell}}\nonumber \\
    &  = -\ta_{\Omega_\ell \backslash \Omega_{\ell-1}} (w_{h,\ell}, {\cR_{h,\ell-1}^\top}\wv_{h,\ell-1}^+ )  -\ri k\langle w_{h,\ell} , \wv_{h,\ell-1}^+ \rangle_{\Gamma_{\ell-1}^+}\nonumber  \\
    &  \qquad +\ri k\langle w_{h,\ell} ,{\cR_{h,\ell-1}^\top} \wv_{h,\ell-1}^+ \rangle_{\partial \Omega_{\ell} \backslash \partial \Omega_{\ell,\ell-1} }, 
                                                                                                 \label{eq:imp-def5}\end{align}
                                                                                               where the last step  used the facts that $\cR_{h,\ell-1}^\top \wv_{h,\ell-1}^+ = \wv_{h,\ell-1}^+$ on $\Omega_{\ell,\ell-1}$ and $\partial \Omega_{\ell, \ell-1} = \left(\partial \Omega_\ell \cap \partial \Omega_{\ell,\ell-1} \right ) \cup \Gamma_{\ell-1}^+$.
Now, to obtain  \eqref{eq:L3143}, we set  
 $w_h = \tcR_{{h,\ell}}^\top w_{h,\ell}$, which is supported in $\overline{\Omega_\ell}$ and satisfies $w_h= w_{h,\ell}$ in $ \mathring{\Omega}_\ell := {\Omega}_\ell \backslash (\Omega_{\ell-1}\cup \Omega_{\ell+1})$. Also $\cR_{h,\ell-1}^\top \wv_{h,\ell-1}^+ $ is supported in $\Omega_{\ell-1} \cup \mathring{\Omega}_\ell$, so    
  $$
  \ta_{\Omega_\ell \backslash \Omega_{\ell-1}} (w_{h}, {\cR_{h,\ell-1}^\top}\wv_{h,\ell-1}^+) =   \ta_{\Omega_\ell \backslash \Omega_{\ell-1}} (w_{h,\ell}, {\cR_{h,\ell-1}^\top}\wv_{h,\ell-1}^+) .
  $$
Moreover, since  $w_h = w_{h,\ell}$ on $\Gamma_{\ell-1}^+$ and $w_h=0$ on $\Gamma_{\ell-1}^-$, 
$$
\langle w_{h,\ell} , \wv_{h,\ell-1}^+\rangle_{\Gamma_{\ell-1}^+}  = \langle w_{h} , \wv_{h,\ell-1}^+\rangle_{\Gamma_{\ell-1}}. 
$$
In a similar way, 
$
\langle w_{h,\ell} , {\cR_{h,\ell-1}^\top} \wv_{h,\ell-1}^+\rangle_{\partial \Omega_{\ell} \backslash \partial \Omega_{\ell,\ell-1} }
=  \langle w_{h} , {\cR_{h,\ell-1}^\top}\wv_{h,\ell-1}^+\rangle_{\partial \Omega \backslash  \partial \Omega_{\ell-1} }.
$
Therefore, using these last three relations in 
\eqref{eq:imp-def5}, and then \eqref{eq:bell}, we obtain
\begin{align*}
  \big\langle  \cI_{\Gamma_{\ell}^s \to \Gamma_{\ell-1}^+}^h g, \wv_{h,\ell-1}^+ \big\rangle_{\Gamma_{\ell-1}^+} & =
                                            -\ta_{\Omega \backslash \Omega_{\ell-1}} (w_{h}, {\cR_{h,\ell-1}^\top} \wv_{h,\ell-1}^+) -\ri k\langle w_{h} , \wv_{h,\ell-1}^+ \rangle_{\Gamma_{\ell-1}} \\
  & \mbox{\hspace{1in}} + \ri k\langle w_{h} ,{\cR_{h,\ell-1}^\top} \wv_{h,\ell-1}^+ \rangle_{\partial \Omega \backslash  \partial \Omega_{\ell-1} } \\
& =  b_{h,\ell-1} ( w_{h} ,   \wv_{h,\ell-1}^+ ) = b_{h,\ell-1} (\tcR_{{h,\ell}}^\top w_{h,\ell} ,   \wv_{h,\ell-1}^+ ),
\end{align*}
as required. 
   \end{proof}

  \begin{theorem}[Connection between $\cT_{h,j,\ell}$ and  $\cI_{\Gamma_{\ell}^s \to \Gamma_{j}^t}^h$] \label{thm:main_T2} Let     
 $w_{h,\ell} \in V_{0,\ell}^s$ with $s\in \{+,-\}$.  Then for any   $(j,t)\in\{(\ell-1,+), (\ell+1,-)\}$, 
 \begin{equation} \label{eq:T20}
  {\rm imp}^h_{\Gamma_{j}^t} \big( \cT_{h; j,\ell} w_{h,\ell}  \big) = \cI_{\Gamma_{\ell}^s \to \Gamma_{j}^t}^h \left(  {\rm imp}^h_{\Gamma_{\ell}^s}  w_{h,\ell}  \right),
  \end{equation}
\end{theorem}

\begin{proof}
  Since $w_{h,\ell}\in V_{0,\ell}^s$, by \eqref{eq:42A} and \eqref{eq:defimpdata}, 
  with 
  $g_h := {\rm imp}^h_{\Gamma_\ell^s} w_{h,\ell}$, 
  we have  
\begin{equation*}
 a_{\Omega_\ell} (w_{h,\ell} , v_{h,\ell}) = \langle g_h ,   v_{h,\ell} \rangle_{\Gamma_{\ell}^s} \quad \tfa \  v_{h,\ell} \in V_\ell.
\end{equation*}
Using  \eqref{H3}/ \eqref{H4} with $g=g_h$ and then  \eqref{eq:def-Tjl}, we have, for any $v_{h,j}\in V_j,$
\begin{equation*}
\big\langle  \cI_{\Gamma_{\ell}^s \to \Gamma_{j}^t}^h g_h, v_{h,j}\big\rangle_{\Gamma_{j}^t} =b_{h, j} (\tcR_{h,\ell}^\top w_{h,\ell} , v_{h,j}) = a_{\Omega_j} (\cT_{h;j,\ell} w_{h,\ell}, v_{h,j}),
\end{equation*}
Since Lemma \ref{lm:Tmaps} implies that $\cT_{h;j,\ell} w_{h,\ell} \in V_{0,j}^t$,  using \eqref{eq:42A}  and \eqref{eq:defimpdata}  again, we have
$$
  \cI_{\Gamma_{\ell}^s \to \Gamma_{j}^t}^h g_h=  {\rm imp}^h_{\Gamma_{j}^t} \left( \cT_{h; j,\ell} w_{h,\ell}  \right).
$$
Then \eqref{eq:T20} follows since  $g_h = {\rm imp}^h_{\Gamma_\ell^s} w_{h,\ell}$. \ \ \end{proof}

\bigskip
The next subsection proves the main result of the paper --Theorem \ref{thm:Tnnorm}-- which shows that if, for some $n$,
$\bcT^n$ is a contraction, then so is $\bcTh^n$ for sufficiently small $h$. This result is obtained by proving that $\Vert   \bcTh^n \Vert_{\mathbb{V}_0} \rightarrow  \Vert \bcT^n\Vert_{\mathbb{U}_0}$. 

\subsection{Convergence of norms }\label{subsec:power}
For $s \in \{+,-\}$, we let $\bGamma^s :=\prod_{\ell=1}^N \Gamma_\ell^s $,     and   $\mathbb{V}(\bGamma^s) := \prod_{\ell = 1}^N V(\Gamma_\ell^s)  \subseteq 
\prod_{\ell=1}^N L^2(\Gamma_\ell^s)= : L^2(\bGamma^s) $, equipped  with the norm
 $$
 \|\bg \|_{L^2(\bGamma^s)}^2 : = \sum_{\ell=1}^N \|g_\ell \|_{L^2(\Gamma_\ell^s)}^2\quad 
\text{for }\bg=(g_1,\cdots,g_N)\in L^2(\bGamma^s). 
 $$
The maps $ \mathbf{imp}^h_s: \mathbb{V}_0 \mapsto \mathbb{V}(\bGamma^s)$ and $\bcI_{s\rightarrow t }^h:L^2(\bGamma^s)\mapsto \mathbb{V}(\bGamma^t)$ are then defined by
  \begin{equation*}
 \mathbf{imp}^h_s:= {\rm diag}\left( {\rm imp}^h_{\Gamma_1^s}, {\rm imp}^h_{\Gamma_2^s},\cdots, {\rm imp}^h_{\Gamma_N^s} \right),
 \end{equation*} 
 and
{\footnotesize \begin{equation*}
 \bcI^h_{s\rightarrow -} 
 :=\begin{pmatrix}
 0			\\
 \cI^h_{\Gamma_1^s\to \Gamma_2^-}	&	0		\\
 		&	\cI^h_{\Gamma_2^s\to \Gamma_3^-}	&0		\\
 		&			&	\ddots	&	\ddots	\\
 		&			&			&	\cI^h_{\Gamma_{N-1}^s\to \Gamma_{N}^-}&0	
 \end{pmatrix} ,  ~ 
 \bcI^h_{s\rightarrow +} 
 :=\begin{pmatrix}
 0		&	\cI^h_{\Gamma_2^s\to \Gamma_1^+} \\
 	&	0		&\cI^h_{\Gamma_3^s\to \Gamma_2^+}\\
 		&			&	\ddots	&	\ddots\\
 		&			&	&		0	&\cI^h_{\Gamma_N^s\to \Gamma_{N-1}^+}\\
 		&			&			&	&0	
 \end{pmatrix},
\end{equation*}
}
where $\imp_{\Gamma_\ell^s}^h$ is defined in Lemma \ref{lm:imp-part}.
The relation \eqref{eq:T20} implies the following corollary.
 \begin{corollary}\label{co:matrix}
 For $\bv_h\in \mathbb{V}_{0}$,  
 \begin{equation}\label{eq:matrix1}
 \mathbf{imp}^h_- \bcL_h \bv_{h} = \bcI_{-\rightarrow -}^h \mathbf{imp}^h_- \bv_{h} + \bcI_{+ \rightarrow -}^h \mathbf{imp}^h_+ \bv_{h}  \quad \text{and}\quad \mathbf{imp}^h_+ \bcL_h \bv_{h} = 0 ,
 \end{equation}
 and 
 \begin{equation}\label{eq:matrix2}
 \mathbf{imp}^h_+ \bcU_h \bv_{h} = \bcI_{- \rightarrow +}^h \mathbf{imp}^h_- \bv_{h} + \bcI_{+ \rightarrow +}^h \mathbf{imp}^h_+ \bv_{h}  \quad \text{and}\quad \mathbf{imp}^h_- \bcU_h \bv_{h} = 0 ,
 \end{equation}
 \end{corollary}
 \begin{proof}
 We only prove \eqref{eq:matrix1}; the proof of   \eqref{eq:matrix2} is similar. The fact that  $\mathbf{imp}^h_+ \bcL_h \bv_{h} = 0 $   follows from 
 Lemmas \ref{lm:Tmaps} and \ref{lm:imp-part}. We now consider the first part  of \eqref{eq:matrix1}, the $\ell$th entry of which is 
 \begin{align} \label{eq:lth}
 {\rm imp}_{\Gamma_\ell^-} \cT_{h,\ell,\ell-1} v_{h,\ell-1} = \cI_{\Gamma_{\ell-1}^-\to \Gamma_\ell^-}^h {\rm imp}_{\Gamma_{\ell-1}^-}^h v_{h,\ell-1} + \cI_{\Gamma_{\ell-1}^+\to \Gamma_\ell^-}^h {\rm imp}_{\Gamma_{\ell-1}^+}^h v_{h,\ell-1},
 \end{align}
 for any $v_{h,\ell-1}\in V_{0,\ell}.$ Now, using \eqref{eq:dec} to write $v_{h,\ell-1} =v_{h,\ell-1}^-+v_{h,\ell-1}^+ $ and Lemma \ref{lm:imp-part}, we obtain \eqref{eq:lth} from Theorem \ref{thm:main_T2}. 
 \end{proof}
 Next we introduce vectorised notation to write \eqref{eq:matrix1} and \eqref{eq:matrix2} in a more compact
 form. Let $\bGamma = \bGamma^-\times \bGamma^+$ and define  $\mathbb{V}(\bGamma) := V(\bGamma^-)\times
   V(\bGamma^+) \subseteq L^2(\bGamma) =: L^2(\bGamma^-) \times L^2(\bGamma^+)$,  with the norm
 \begin{equation*}
   \|\bg\|_{L^2(\bGamma)}^2 : =  \|\bg^-\|_{L^2(\bGamma^-)}^2+\|\bg^{+}\|_{L^2(\bGamma^+)}^2\quad \tfa \
   \bg = (\bg^{-},\bg^{+}) \in L^2(\bGamma).
 \end{equation*}
We then define the maps $ \mathbf{imp}^h: \mathbb{V}_0 \mapsto \mathbb{V}(\bGamma)$ and $\bcI^h:L^2(\bGamma)\mapsto \mathbb{V}(\bGamma)$  by
 $$
 \mathbf{imp}^h = \begin{pmatrix}
 \mathbf{imp}_-^h\\
 \mathbf{imp}_+^h
 \end{pmatrix}
\quad \text{and}\quad 
\bcI^h =  \begin{pmatrix}
\bcI^h_{-\rightarrow-}& \bcI^h_{+\rightarrow -}\\
\bcI^h_{- \rightarrow +}& \bcI^h_{+\rightarrow +}
\end{pmatrix} .
$$
Then we have the following theorem 
\begin{theorem}\label{thm:Th-Ih}
For any $\bv_h\in \mathbb{V}_0,$
\begin{equation}\label{eq:Th-Ih}
\mathbf{imp}^h\left( \bcT_h \bv_h \right) = \bcI^h \left(\mathbf{imp}^h \bv_h \right).
\end{equation}
\end{theorem}
\begin{proof}
By \eqref{eq:splitT}, $\bcT_h = \bcL_h +\bcU_h$. By Lemma \ref{lm:Tmaps},  $\mathbf{imp}_-^h \bcU_h \bv_h = \mathbf{0} = \mathbf{imp}_+^h \bcL_h \bv_h $, and thus
\begin{align}\label{eq:vector}
\mathbf{imp}^h\left( \bcT_h \bv_h \right)  =\begin{pmatrix}
\mathbf{imp}^h_-\left( \bcT_h \bv_h \right) \\
\mathbf{imp}^h_+\left( \bcT_h \bv_h \right) 
\end{pmatrix}
=\begin{pmatrix}
\mathbf{imp}^h_-\left( \bcL_h \bv_h \right) \\
\mathbf{imp}^h_+\left( \bcU_h \bv_h \right);
\end{pmatrix}
\end{align}
the result then follows by combining \eqref{eq:vector} with the first equations in \eqref{eq:matrix1} and \eqref{eq:matrix2}.
\end{proof}
Using \eqref{eq:Th-Ih} $n$ times, we have the following corollary.
\begin{corollary}\label{co:Th-Ihn}
For any integer $n\geq 1$,
\begin{equation}\label{eq:Th-Ihn}
\mathbf{imp}^h\left( \bcT_h^n \bv_h \right) = \left(\bcI^h \right)^n\left(\mathbf{imp}^h  \bv_h\right) \quad \text{for any $\bv_h\in \mathbb{V}_0$}.
\end{equation}
\end{corollary}
The next lemma gives an alternative expression for the norm $\Vert \cdot \Vert_{\mathbb{V}_0}$.
\begin{lemma}\label{lm:v0Norm}
For any $\bw_h \in \mathbb{V}_0$, 
\begin{equation}\label{eq:v0Norm}
\|\bw_h\|_{\mathbb{V}_0} = \|\mathbf{imp}^h \bw_h\|_{L^2(\bGamma)}.
\end{equation}
\end{lemma}
\begin{proof}
Let $\bw_h = (w_{h,1},\cdots,w_{h,N}) \in \mathbb{V}_0$, and use Lemma \ref{prop:decomp} to write   $w_{h,\ell}=w_{h,\ell}^- + w_{h,\ell}^+ $. By \eqref{eq:defimpnorm},  
\begin{align*}
  \|\bw_h\|_{\mathbb{V}_0}^2  &= \sum_\ell \|w_{h,\ell}\|_{V_{0,\ell}}^2
 = \sum_\ell \|\imp_{\Gamma_\ell}^h w_{h,\ell}\|_{V_{0,\ell}}^2.  
 \end{align*}
Then, using Lemmas \ref{prop:decomp} and \ref{lm:imp-part}, we obtain  
\begin{align*}
\|\bw_h\|_{\mathbb{V}_0}^2  & = \sum_\ell \Vert \imp_{\Gamma_\ell^-}^h w_{h,\ell}^-\Vert_{L^2(\Gamma_\ell^-)}^2 +  \sum_\ell \Vert \imp_{\Gamma_\ell^+}^h w_{h,\ell}^+\Vert_{L^2(\Gamma_\ell^+)}^2 \\ & = \Vert  \mathbf{\imp}_{-}^h {\bw}_h\Vert_{L^2(\bGamma^-)}^2 +  \Vert \mathbf{\imp}_{-}^h {\bw}_h\Vert_{L^2(\bGamma^+)}^2 = \Vert \mathbf{\imp}^h {\bw}_h\Vert_{L^2(\bGamma)}^2.  
             \end{align*}   \end{proof}

\begin{lemma}\label{lm:Th-IhnNorm}
For any integer $n\geq 1$,
\begin{equation}\label{eq:Th-IhnNorm}
\|\bcT_h^n\|_{\mathbb{V}_0} = \left\| \left(\bcI^h \right)^n \right\|_{L^2(\bGamma)}.
\end{equation}
\end{lemma}
\begin{proof}
  By Lemma \ref{lm:v0Norm} and Corollary \ref{co:Th-Ihn},
$$\begin{aligned}
\|\bcT_h^n\|_{\mathbb{V}_0} & = \sup_{\bw_h \in \mathbb{V}_0} \frac{\|\bcT_h^n \bw_h\|_{\mathbb{V}_0}  }{\|\bw_h\|_{\mathbb{V}_0}} =  \sup_{\bw_h \in \mathbb{V}_0} \frac{\| \mathbf{imp}^h \left(\bcT_h^n \bw_h\right) \|_{L^2(\bGamma) } }{\|\mathbf{imp}^h~\bw_h\|_{L^2(\bGamma)}}\\
& =  \sup_{\bw_h \in \mathbb{V}_0} \frac{\|  (\bcI^h)^n\left( \mathbf{imp}^h~ \bw_h\right) \|_{L^2(\bGamma) } }{\|\mathbf{imp}^h~\bw_h\|_{L^2(\bGamma)}}
\  \le \ \left\| \left(\bcI^h \right)^n \right\|_{L^2(\bGamma)}. 
\end{aligned}
$$
The lemma then follows if we can show that  $\|\bcT_h^n\|_{\mathbb{V}_0} \ge  \| (\bcI^h)^n \|_{L^2(\bGamma)}$.

For each $\ell,s$, let $P^h_{\Gamma_\ell^s}$ denote the  $L^2$ orthogonal projection from $L^2(\Gamma_\ell^s)$ to $V(\Gamma_\ell^s)$,  and let    $\boldsymbol{P}^h = \mathrm{diag}(P^h_{\Gamma_1^-}, \ldots , P^h_{\Gamma_N^-},   P^h_{\Gamma_1^+}, \ldots , P^h_{\Gamma_N^+})$. A simple examination of Definition \ref{def:impmap} shows that
  $\cI_{\Gamma_\ell^s\rightarrow \Gamma_j^t} = \cI_{\Gamma_\ell^s\rightarrow \Gamma_j^t} P_{\Gamma_\ell^s}^h$,  and so
\begin{equation}\label{eq:Th-IhnNorm1}
\bcI^h \bg =\bcI^h \boldsymbol{P}^h \bg \quad \tfa  \ \bg \in L^2(\bGamma).
\end{equation}

Let $\bg = (\bg^-,\bg^+)$ with $\bg^s = (g_{\Gamma_1^s}, g_{\Gamma_2^s},\cdots,g_{\Gamma_N^s}) \in \mathbb{V}(\bGamma^s)$.  Let the components of $\bw_h  \in \mathbb{V}_0$ be defined as the solutions of
$$
a_{\ell}(w_{h,\ell},v_{h,\ell}) = \langle g_{\Gamma_\ell^-}, v_{h,\ell}\rangle_{\Gamma_\ell^-} + \langle g_{\Gamma_\ell^+}, v_{h,\ell}\rangle_{\Gamma_\ell^+}\quad \tfa \  v_{h,\ell}\in V_\ell.
$$
Then, by Lemma  \ref{lm:iso},  
\begin{equation}\label{eq:Th-IhnNorm2}
\mathbf{imp}^h \bw_h = \boldsymbol{P}^h \bg.
\end{equation}
Thus, by using  \eqref{eq:Th-IhnNorm1}-\eqref{eq:Th-IhnNorm2} and \eqref{eq:Th-Ihn}, \eqref{eq:v0Norm}, we have
$$
\begin{aligned}
\left\| \left(\bcI^h \right)^n \right\|_{L^2(\bGamma)} &= \sup_{\bg\in L^2(\bGamma)} \frac{\| \left(\bcI^h\right)^n \bg \|_{L^2(\bGamma)} }{\|\bg\|_{L^2(\bGamma)}}\le  \sup_{\bg\in L^2(\bGamma)} \frac{\| \left(\bcI^h\right)^n \boldsymbol{P}^h\bg \|_{L^2(\bGamma)} }{\| \boldsymbol{P}^h\bg\|_{L^2(\bGamma)}}\\
&  \le   \sup_{\bw_h \in \mathbb{V}_0} \frac{\|\left(\bcI^h \right)^n\left(\mathbf{imp}^h \bw_h \right)\|_{L^2(\bGamma)} }{\|\mathbf{imp}^h \bw_h \|_{L^2(\bGamma)}} 
 = \sup_{\bw_h \in \mathbb{V}_0} \frac{\|\mathbf{imp}^h \left(\bcT_h^n \bw_h \right)\|_{L^2(\bGamma)} }{\|\mathbf{imp}^h \bw_h \|_{L^2(\bGamma)}} 
= \|\bcT_h^n\|_{\mathbb{V}_0}.
\end{aligned}
$$
\end{proof}
The analogous  results to Theorem \ref{thm:Th-Ih}, Corollary \ref{co:Th-Ihn} and Lemmas \ref{lm:v0Norm} and \ref{lm:Th-IhnNorm} can be easily established at  the continuous level. Define the continuous vectorised maps $ \mathbf{imp}_s: \mathbb{U}_0 \mapsto L^2(\bGamma^s)~  \text{and } \bcI_{s\rightarrow t}:L^2(\bGamma^s)\mapsto L^2(\bGamma^t),$ as follows,
 \begin{equation*}
 \mathbf{imp}_s:= {\rm diag}\left( {\rm imp}_{\Gamma_1^s}, {\rm imp}_{\Gamma_2^s},\cdots, {\rm imp}_{\Gamma_N^s} \right),
 \end{equation*} 
where $\imp_{\Gamma_\ell^s}v_\ell = (\partial/\partial n_\ell - \ri k ) v_\ell$ on $\Gamma_\ell^s$, and 
  \begin{equation}\label{eq:matrix-opts}
    {\footnotesize
      \bcI_{s\rightarrow -} 
 :=\begin{pmatrix}
 0			\\
 \cI_{\Gamma_1^s\to \Gamma_2^-}	&	0		\\
 		&	\cI_{\Gamma_2^s\to \Gamma_3^-}	&0		\\
 		&			&	\ddots	&	\ddots	\\
 		&			&			&	\cI_{\Gamma_{N-1}^s\to \Gamma_{N}^-}&0	
 \end{pmatrix},  ~ 
 \bcI_{s\rightarrow +} 
 :=\begin{pmatrix}
 0		&	\cI_{\Gamma_2^s\to \Gamma_1^+} \\
 	&	0		&\cI_{\Gamma_3^s\to \Gamma_2^+}\\
 		&			&	\ddots	&	\ddots\\
 		&			&	&		0	&\cI_{\Gamma_N^s\to \Gamma_{N-1}^+}\\
 		&			&			&	&0	
 \end{pmatrix}.
}
\end{equation}
Then define the maps $ \mathbf{imp}: \mathbb{U}_0 \mapsto L^2(\bGamma)$ and $\bcI :L^2(\bGamma)\mapsto L^2(\bGamma)$ By
 $$
 \mathbf{imp} = \begin{pmatrix}
 \mathbf{imp}_-\\
 \mathbf{imp}_+
 \end{pmatrix}
\quad \text{and}\quad 
\bcI =  \begin{pmatrix}
\bcI_{-\rightarrow -}& \bcI_{+\rightarrow -}\\
\bcI_{-\rightarrow +}& \bcI_{+\rightarrow +}
\end{pmatrix} .
$$
\begin{lemma}
For any integer $n\geq 1$,
\beqs
\mathbf{imp} \left( \bcT^n \bv \right) = \bcI^n \left( \mathbf{imp}~ \bv \right)\quad \text{for any $\bv\in \mathbb{U}_0$},
\eeqs
and 
\begin{equation}\label{eq:T-InNorm}
\|\bcT^n\|_{\mathbb{U}_0} = \left\| \bcI ^n \right\|_{L^2(\bGamma)}.
\end{equation}
\end{lemma}

\begin{lemma}
\begin{equation}\label{eq:Ih2I-vec}
\|\bcI^h -\bcI\|_{L^2(\bGamma)} \to 0 \quad \text{as } h\to 0.
\end{equation}
Moreover, there exist a constant $C>0$ such that, for any $\bg\in L^2(\bGamma)$ and sufficiently small $ h$,
\begin{equation}\label{eq:Ibounded}
\|\bcI\bg \|_{L^2(\bGamma)}\le C \|\bg\|_{L^2(\bGamma)}\quad \text{and}\quad
\|\bcI^h \bg \|_{L^2(\bGamma)}\le C\|\bg\|_{L^2(\bGamma)}.
\end{equation}
 \end{lemma}
\begin{proof}
\eqref{eq:Ih2I-vec} is a direct corollary of Theorem \ref{thm:imp-map-cont}. The boundedness of $\bcI$ follows from \cite[Lemma 3.8]{GoGaGrLaSp:21}. The boundedness of $\bcI^h$ follows from 
the boundedness of $\bcI$ and \eqref{eq:Ih2I-vec}.
\end{proof}

Finally we obtain our main result. 
\begin{theorem}\label{thm:Tnnorm} 
 For any fixed  integer $n\geq 1,$ 
\begin{equation*}
\|\bcT_h^n\|_{\mathbb{V}_0} \to \|\bcT^n\|_{\mathbb{U}_0}\quad \text{as }h \to 0.
\end{equation*}
\end{theorem}
\begin{proof}
By \eqref{eq:Th-IhnNorm} and \eqref{eq:T-InNorm}, and  the reverse  triangle inequality, 
$$
\Big|\|\bcT_h^n\|_{\mathbb{V}_0} - \|\bcT^n\|_{\mathbb{U}_0} \Big| = \Big\vert \left\Vert \left(\bcI^h\right)^n \right\Vert_{L^2(\bGamma)} -  \left\Vert \bcI^n \right\Vert_{L^2(\bGamma)}\Big\vert    \le \left\|\left( \bcI^h\right)^n  - \bcI^n \right\|_{L^2(\bGamma)}.
$$
So it is sufficient to show that 
$
\|( \bcI^h)^n  - \bcI^n \|_{L^2(\bGamma)} \to 0\quad \text{as} ~h\to 0.$
We now prove by induction that
\begin{equation}\label{eq:Tnnorm1}
\left\|\left( \bcI^h\right)^n  - \bcI^n \right\|_{L^2(\bGamma)}\le nC^{n-1}  \left\| \bcI^h  - \bcI \right\|_{L^2(\bGamma)},
\end{equation}
where $C$ is as in \eqref{eq:Ibounded}; the result then follows by combining \eqref{eq:Tnnorm1} with \eqref{eq:Ih2I-vec} (for any fixed $n$).
 First observe that \eqref{eq:Tnnorm1} obviously holds for $n=1$. Assuming it holds $n=m-1$ and using \eqref{eq:Ibounded}, we have, for $n=m,$
$$\begin{aligned}
\left\|\left( \bcI^h\right)^m  - \bcI^m \right\|_{L^2(\bGamma)} &=\left\| \left( \left(\bcI^h \right)^{m-1} - \bcI^{m-1} \right)\bcI^h  + \bcI^{m-1}\left(\bcI^h - \bcI\right) \right\|_{L^2(\bGamma)} \\
&\le \left\| \left( \left(\bcI^h \right)^{m-1} - \bcI^{m-1} \right)\bcI^h  \right\|_{L^2(\bGamma)} 
  + \left\| \bcI^{m-1}\left(\bcI^h - \bcI\right) \right\|_{L^2(\bGamma)} \\
  &\le (m-1)C^{m-2} \left\| \bcI^h  - \bcI \right\|_{L^2(\bGamma)} C
  + C^{m-1}\left\| \bcI^h - \bcI \right\|_{L^2(\bGamma)}\\
  & = mC^{m-1}\left\| \bcI^h - \bcI \right\|_{L^2(\bGamma)};
\end{aligned}
$$
thus \eqref{eq:Tnnorm1} holds for any integer $n\geq 1$ and the proof is complete.
\end{proof}

\subsection{Power contractivity of the ORAS iteration}
\label{subsec:power_recap}

If $\bcT^n$ is a contraction, then Theorem \ref{thm:Tnnorm} has  the following corollary.
\begin{corollary} \label{cor:Tnnorm}
  Suppose $\Vert \bcT^n\Vert_{\mathbb{U}_0} < C <1$ and suppose ORAS is implemented using finite elements of degree $p$.
  Then for all $\varepsilon >0$ there exists $h_0 = h_0(p,\varepsilon) >0$,  such that
  $$\Vert \bcTh^n\Vert_{\mathbb{V}_0} < C+ \varepsilon  \quad \tfa \ h \leq h_0(p,\varepsilon). $$
\end{corollary}
That is, the power  contractive property of ORAS is independent of $p$ and $h$ for $h$ sufficiently small (with a $p$-dependent threshold).   

In the rest of this section, we discuss the results of \cite{GoGaGrLaSp:21},  giving insight into when the condition $\Vert \bcT^n\Vert_{\mathbb{U}_0} < C <1$ holds.
The paper \cite{GoGaGrLaSp:21}
studies the power contractivity of   the operator $\bcT$ by using the decomposition $\bcT = \bcL + \bcU$ into lower and upper triangular parts
  (analogous to \eqref{eq:splitT}). Analogously  to  Lemma \ref{lm:Tmaps},  $\bcL$ and $\bcU$ are maps
on vectors of Helmholtz-harmonic functions   and    
the products $\bcL\bcU$,  $\bcU\bcL$ can be characterised, respectively, in terms of
the action of right-to-left and left-to-right  impedance-to-impedance  maps.  Computations in \cite{GoGaGrLaSp:21} show that impedance-to-impedance maps that switch direction are typically  $\ll 1$ in norm, while impedance maps that preserve direction have norm very close to $1$.  With 
$\rho$  denoting the maximum of the norms of the impedance-to-impedance
maps that switch direction, and $\gamma$ denoting
the maximum  of the norms of the maps that preserve direction, it is proved rigorously in \cite{LaSp:21} that,
for a particular model problem,  $\rho$ decreases with $\mathcal{O}(\delta^{-2})$, for $k$ sufficiently large  where $\delta$ is the width of  the overlap of neighbouring subdomains. Also, since the norm of any discrete impedance map can be
computed (by solving  a local eigenvalue problem) and since the  discrete map converges in
norm to the continuous map, we can also compute $\rho$, $\gamma$, with guaranteed  accuracy,  for moderate $k$.  Computations in
\cite{GoGaGrLaSp:21} show that typically that $\bcT^n$ is a contraction when $n \geq N$ and that
$\rho$ is decreases rapidly with increasing  $\delta$ while  $\gamma\approx 1$.
Thus, 
upper bounds on $\Vert \bcT^n\Vert_{\mathbb{U}_0}$ in terms of $\rho$ and $\gamma$, where the upper bound $\to 0$ as $\rho \to 0$, help justify observations of power contractivity of  $\bcT^n$.
A key result in \cite{GoGaGrLaSp:21} is that, for any $n \geq N$,  
  \begin{align} \label{eq:composite} \Vert \bcT^n\Vert_{\mathbb{U}_0} \ \leq \ 2 \sum_{j=1}^{n-1} \left( \begin{array}{c} n-1\\j \end{array} \right) \max_{p\in \cP(n,j)} \Vert p(\cL, \cU) \Vert_{\mathbb{U}_0},
  \end{align}
  where $\cP(n,j)$ is the set of monomials  of degree $n$  in 2 variables with the property
that $p(\bcL,\bcU)$ contains $j$ switches from  $\bcL$ to  $\bcU$ or $\bcU$ to  $\bcL$. Hence  terms in \eqref{eq:composite}  corresponding to $j=1$ 
contain one impedance map that switches direction and $N-1$ maps that  preserve direction,  leading to
the estimate: 
\begin{align}\label{eq:Tcontract}
  \|\boldsymbol{\mathcal{T}}^N  \|_{\mathbb{U}_0} \  \le \ 4   \gamma^{N-1} (N-1) {\rho }  \ + \ \mathcal{O}({\rho^2 }). 
\end{align}
While \eqref{eq:Tcontract} shows that, for fixed $N$,  $\cT^N$ is contracting when $\rho$ is  small enough,  computations in
\cite{GoGaGrLaSp:21} suggest that the power contractive property is   independent  of $N$,
for fixed $\rho$. This property is  explained in \cite{GoGaGrLaSp:21}   by
estimating terms  $\Vert p(\bcL,\bcU)\Vert_{\mathbb{U}_0}$ for $p \in \cP(N,1)$ more carefully in terms of norms of {\em compositions} of
impedance-to-impedance maps; see \cite[\S 4]{GoGaGrLaSp:21}.
The analysis in \cite[\S 4]{GoGaGrLaSp:21} also obtains refinements of \eqref{eq:Tcontract} that explain practical observations \cite[Figure 6.1]{GoGaGrLaSp:21} that the convergence profile of the algorithm takes a sharp jump downwards after each batch of $N$ iterations.
Finally we highlight that \cite{GoGaGrLaSp:21} also contains extensive experiments on non-strip-type domain decompositions, and these indicate that power contractivity persists in this more-general situation.

     \section{Convergence  of  discrete impedance-to-impedance maps}\label{sec:imp}
     In this section we     prove Theorem \ref{thm:imp-map-cont}, thus establishing the norm  convergence of the discrete impedance-to-impedance maps to their continuous counterparts.

     To reduce notational technicalities, we note that the proof is  the same for all $\ell$. Moreover, by symmetry, if we prove Theorem \ref{thm:imp-map-cont} for $s$ equal to $-$ then we also have the result for $s$ equal to $+$. Furthermore, for any choice of $s$,  the proofs of the two  convergence results in \eqref{eq:impmap_conv} are almost identical, so we just prove the first one.
     Finally by the affine change of coordinates $x  = H\tilde{x}$, a Helmholtz problem on any of the domains $\Omega_\ell$ in Figure \ref{fig:overlap} can be scaled to a Helmholtz problem on the rectangle $[0,L]\times [0,1]$, for some $L$, leading only to a multiplicative change in the wavenumber $k$.
     Thus in this section, without loss of generality we restrict to the canonical domain $\Omega = [0,L] \times [0,1]$ depicted in Figure \ref{fig:hat}. We study only the left-to-right impedance map (which we denote  $\cI$),  which  takes,  as input,   left-facing impedance data on $\Gamma^-$, solves the Helmholtz problem on $\Omega$
     and evaluates the right-facing impedance data at the vertical  interface  $\Gamma_\delta$, which is situated a distance $\delta \in (0,L)$ from $\Gamma^-$.   In  the case of a  physical subdomain $\Omega_\ell$ (see Figure \ref{fig:overlap}), $\Gamma^-$ corresponds to  $\Gamma_{\ell}^-$ and   $\Gamma_\delta$ corresponds to  $\Gamma_{\ell-1}^+$ .

\begin{figure}[H]
  \begin{center}
  \includegraphics[scale=0.8]{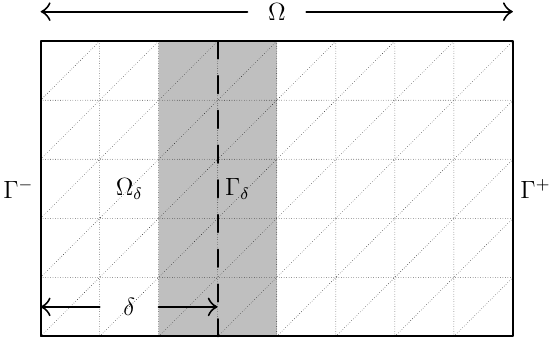}
\end{center}   
\caption{The canonical  domain ${\Omega}$, and the domain $\Omega_\delta$  bounded by  $\Gamma^-$ and $\Gamma_\delta$.  
 \label{fig:hat}}
\end{figure}

For clarity,  we now  redefine the map  $\cI$, as well as   its finite element counterpart $\cI^h$, essentially rewriting the definitions given in \eqref{eq:ctsimp},  \eqref{eq:imp-def2} and  \eqref{eq:imp-def3},   but using the simplified
  setting of    Figure \ref{fig:hat}. Throughout this section (only) we need the
  following notations, which are equivalent to  notations used previously, but  recast in the simpler setting of the 
  canonical domain. 
  
  \begin{notation}[Notation specific to  canonical domain] \label{not:canon}
    $V = V(\Omega)$ denotes the finite element space  on $\Omega$ (with mesh $T^h$,  of diameter $h$ and   assumed to resolve the interior interface $\Gamma_\delta$).  $V(\Omega_\delta)$ and $V(\Gamma_\delta)$ denote the  restrictions of $V$  to
    $\Omega_\delta$ and $\Gamma_\delta$ respectively. For any $v_\delta \in H^1(\Omega_\delta)$, $\cR_\delta^\top v_\delta $ denotes its zero extension to all of $\Omega$ (analogous to \eqref{eq:zero-ext}). For any $v_{h,\delta} \in V(\Omega_\delta)$, $\cR^\top_{h,\delta} v_{h,\delta} \in V$ denotes its zero nodal extension to all of $\Omega$ (analogous to \eqref{nodewise}). For any $v_{h,\Gamma_\delta} \in V(\Gamma_\delta)$, $\wv_{h,\delta}\in V(\Omega_\delta)$ denotes its zero nodal extension to all of $\Omega_\delta$ (analogous to Notation \eqref{not:ext}(i)).    We also need to pay special attention to the mesh elements  that touch the interface $\Gamma_\delta$ (see the shaded region in Figure \ref{fig:hat}),  and so we  set
    \begin{align} \label{eq:Odelta} \Omega_\delta' := \bigcup \big\{ \tau \in T^h: {\tau} \cap \Gamma_\delta \not=0\big\}, \end{align} where the elements $\tau$ are assumed closed.
  \end{notation}
 \begin{assumption}[Mesh assumption]\label{ass:mesh}   $T^h$ is  shape-regular.
In addition, the
 elements touching the interface  $\Gamma_\delta$ are quasiuniform in the sense that 
 there exists a constant $C_0\geq 1$ such that  
   \begin{align} h \leq C_0 h_\tau \quad \tfa \  \tau \subset \Omega_\delta'. \label{eq:localqu}\end{align}    
   \end{assumption}

  \begin{definition}[$\cI$ and $\cI^h$ defined in canonical domain]\label{def:imp-can}
     Given $g\in L^2(\Gamma^-)$, let $u\in H^1(\Omega)$ and $u_h\in V$ be the solutions of 
     \begin{equation}
     \label{eq:imp-map-sol}
     a(u,v) = \langle g, {v}\rangle_{\Gamma^{-}}\quad \text{for all} \quad v \in H^1(\Omega),
     \end{equation}
      \begin{equation}
     \label{eq:discrete-imp-map-sol}
     a(u_h,v_h) =  \langle g, {v_h}\rangle_{\Gamma^{-}} \quad \text{for all}\quad  v_h \in V(\Omega).
     \end{equation}
     The continuous left-to-right impedance map $\cI:L^2(\Gamma^-)\to L^2(\Gamma_\delta)$ is defined (analogously to \eqref{eq:ctsimp})  as
     \begin{equation*}
     \cI g = \left. \left(\left(  \frac{\partial}{\partial x_1}  -\ri k  \right)u \right)\right|_{\Gamma_\delta} \in L^2(\Gamma_\delta)
     \end{equation*}
The discrete counterpart of $\cI$ is  $\cI^h:L^2(\Gamma^-) \to V(\Gamma_\delta)$,  defined (analogously to \eqref{eq:imp-def2})  by
     \begin{equation}\label{eq:imp-map-discrete}
     \big\langle  \cI^h g , v_{h,\Gamma_\delta} \big\rangle_{\Gamma_\delta} := a_{\Omega_\delta} (u_h,  \wv_{h,\delta} ) -  a(u_h,  \CR^\top_{h,\delta} \wv_{h,\delta} )\quad \text{for all}\quad v_{h,\Gamma_\delta} \in V(\Gamma_\delta),
   \end{equation}
     \end{definition}
     
The rest of this section is devoted to proving  Theorem \ref{thm:main} below, which   is  
sufficient to establish Theorem  \ref{thm:imp-map-cont}.
From here on our estimates are explicit in the parameters $k$ and $h$.  
In this context, we use the following notation. 
      \begin{notation}\label{notation}
We write   $A \lesssim B $ if $A\leq CB$ where $C$ is  independent of $h$ and $k$. We write $A \sim B$ when $A \lesssim B$ and $B \lesssim A$.   Since the entire theory of this paper is underpinned by the condition \eqref{eq:defhk} and since we are studying the limit  $h \rightarrow 0$ we assume, without loss of generality, that
\begin{align} hk \leq 1 \quad \text{and} \quad k \geq k_0 >0.  \label{eq:hkone}\end{align}
\end{notation}

 \begin{theorem}[Norm convergence of discrete imp-to-imp maps]
 \label{thm:main} Suppose Assumption \ref{ass:mesh} holds and
   $h^{1/2} k$ is sufficiently small. Then, 
 \begin{equation*}
  \Vert \cI - \cI^h \Vert \ \lesssim  \ h^{1/2}k  + h k^3 . 
  \end{equation*}
  \end{theorem} 

 \subsection{Proof of Theorem \ref{thm:main}}
 In the following we use the Helmholtz energy inner product for any subdomain
 $\Omega'\subset \Omega$,   
   \begin{align} \label{eq:Helmen}
     (v, w)_{1,k,\Omega'} := (\nabla v, \nabla w)_{\Omega'} + k^2 (v,w)_{\Omega'},
   \end{align}
   and we denote the induced norm by $\Vert \cdot \Vert_{1,k,\Omega'}$; when $\Omega' = \Omega$ we just write $(\cdot, \cdot)_{1,k}$ and $\Vert \cdot \Vert_{1,k}$. 
   
 Theorem  \ref{thm:main} is a direct corollary of the following two  results  (by setting $\alpha=1/2$ in Theorem \ref{thm:main_weighted}). 
 \begin{lemma} \label{th:65}
If Assumption \ref{ass:mesh} holds, then
\begin{align} \label{eq:th65}
  \|\cI g - \cI^h g\|_{L^2(\Gamma_\delta)}  \ \lesssim \ h^{1/2} k \Vert g \Vert_{L^2(\Gamma^{-})} +
  h^{-1/2} \Vert u - u_h \Vert_{1,k, \Omega_\delta'}.\end{align} 
\end{lemma}

\begin{theorem} \label{thm:main_weighted} 
Suppose that Assumption \ref{ass:mesh} holds and that, given $\alpha \in [0,1/2]$, 
$h^{2-\alpha} k^3$ is sufficiently small.  Then
  \begin{align*}
   \Vert u - u_h \Vert_{1,k,\Omega_\delta'} \ \lesssim  \big(h^{\alpha+1/2}k +h^{3/2}k^3\big) \Vert g \Vert_{L^2(\Gamma^-)}.
   \end{align*} 
\end{theorem}
Lemma \ref{th:65} is proved at the end of this subsection. The proof of Theorem \ref{thm:main_weighted} 
is more involved,  and is  postponed until  \S\ref{subsec:weighted}. 
Although Theorem \ref{thm:main_weighted} is stated and proved for $\Omega$ given in Figure \ref{fig:hat}, appropriate analogues of it hold for more general geometries; see Remark \ref{rem:more_general_geometries}.

 First we recall
 that since $\delta> 0$,  the regularity of the solution $u$ of \eqref{eq:imp-map-sol}
 with respect to the data $g$ is much better in a neighbourhood of the
  interface $\Gamma_\delta$ than it is near the part of the boundary $\Gamma^-$ where the
  impedance condition involving $g$ is imposed. 
  \begin{lemma}\label{lm:reg}
If $u$ is the solution of \eqref{eq:imp-map-sol}, then, given $\beta>1/2$,
  \begin{equation}\label{eq:reg-1.5}
  \|u\|_{1,k}\lesssim \|g\|_{L^2(\Gamma^-)}\quad \text{and}\quad
  \|u\|_{H^{3/2}(\Omega)} \lesssim k^\beta \|g\|_{L^2(\Gamma^-)}.
  \end{equation}
If $D$ be a convex polygonal subdomain of $\Omega$ with  ${\rm dist}(D, \Gamma^-) \gtrsim 1$, then
    \begin{equation}\label{eq:reg-2}
  \|u\|_{H^{2}(D)} \lesssim k\|g\|_{L^2(\Gamma^-)}.
  \end{equation}
  \end{lemma}
\begin{proof}
The first inequality of \eqref{eq:reg-1.5} follows from \cite[Lemma 2.4]{GoGaGrLaSp:21} with $f=0$. The second inequality of \eqref{eq:reg-1.5} follows from \cite[Theorem 2.9]{GoGaGrLaSp:21}  with $f=0$ 
 there. To prove \eqref{eq:reg-2}, first observe that if $\dist({D'},\partial \Omega)>0$, then by interior elliptic regularity for the Laplacian (see, e.g., \cite[Theorem 4.16]{Mc:00}), and $\Delta u =k^2 u $, 
\beqs
\N{u}_{H^2({D'})}\lesssim \N{\Delta u }_{L^2(D)} + \N{u}_{H^1(D)} \ \lesssim \  k^2 \|u\|_{L^2(D)} + \N{u}_{H^1(D)}; 
\eeqs 
the bound \eqref{eq:reg-2} in this case then follows from the first inequality in \eqref{eq:reg-1.5}.

It is therefore sufficient to prove \eqref{eq:reg-2} for ${D}$ supported in a neighbourhood of $\partial \Omega\backslash \Gamma^s$.
In this case, by elliptic regularity up to the boundary for the Laplacian (see, e.g., \cite[Theorem 4.18]{Mc:00}), using the fact that $D$ is convex,
\begin{align*}
\N{u}_{H^2({D})}&\lesssim \N{\Delta u }_{L^2(D)} + \N{u}_{H^1(D)} + \N{\partial u/\partial n}_{H^{1/2}(\partial \Omega \cap D)} \\
&=  k^2\N{ u }_{L^2(D)} + \N{u}_{H^1(D)} + k\N{u}_{H^{1/2}(\partial \Omega \cap D)},
\end{align*}
The bound \eqref{eq:reg-2} then follows from the trace result 
$\Vert u \Vert_{H^{1/2}(\Gamma_2)} \ \leq\    \|u\|_{H^{1/2}(\partial \Omega)}\lesssim \|u\|_{H^1(\Omega)}$
and the first inequality in \eqref{eq:reg-1.5}.
\end{proof}

\begin{lemma}\label{lm:strang}
Given $g \in L^2(\Gamma^{-})$, let $u$ and $u_h$ be the solutions of \eqref{eq:imp-map-sol} and \eqref{eq:discrete-imp-map-sol} respectively. Then
\begin{equation}\label{eq:strang}
\|\cI g - \cI ^h g\|_{L^2(\Gamma_\delta)} \le \inf_{v_{h,\Gamma_\delta} \in {V}^h(\Gamma_\delta)} \|\cI  g - v_{h,\Gamma_\delta} \|_{L^2(\Gamma_\delta)} + \sup_{v_{h,\Gamma_\delta} \in {V}^h(\Gamma_\delta)\bzero } \frac{\vert a_{\Omega_\delta}(u - u_h, \wv_{h,\delta}   )\vert}{ \| v_{h,\Gamma_\delta}\|_{L^2(\Gamma_\delta)}},
\end{equation}
where $\wv_{h,\delta}$ is defined in Notation \ref{not:canon}.
\end{lemma}

\begin{proof}
Let $\cP_{\Gamma_\delta}^h:L^2(\Gamma_\delta) \mapsto {V}^h(\Gamma_\delta)$  denote the $L^2$-orthogonal projection onto  ${V}^h(\Gamma_\delta)$ and set $z_h = \cP_{\Gamma_\delta}^h \cI g$. Then, since $z_h - \cI^hg \in V(\Gamma_\delta)$, 
\begin{align*}
\|\cI g    - \cI^h g\|_{L^2(\Gamma_\delta)} &\le \|\cI g - z_h\|_{L^2(\Gamma_\delta)} + \|z_h - \cI^h g \|_{L^2(\Gamma_\delta)}\nonumber \\
&\mbox{\hspace{-1.5cm}}= \inf_{v_{h,\Gamma_\delta} \in {V}^h(\Gamma_\delta)} \|\cI  g - v_{h,\Gamma_\delta} \|_{L^2(\Gamma_\delta)} + \sup_{v_{h,\Gamma_\delta} \in {V}(\Gamma_\delta)\bzero } \frac{ \big|\langle z_h- \cI^h g , v_{h,\Gamma_\delta} \rangle_{\Gamma_\delta}\big|}{ \| v_{h,\Gamma_\delta}\|_{L^2(\Gamma_\delta)}} \nonumber\\
&\mbox{\hspace{-1.5cm}}=  \inf_{v_{h,\Gamma_\delta} \in {V}^h(\Gamma_\delta)} \|\cI g - v_{h,\Gamma_\delta} \|_{L^2(\Gamma_\delta)} + \sup_{v_{h,\Gamma_\delta} \in {V}^h(\Gamma_\delta)\bzero } \frac{ \big|\langle \cI g - \cI ^h g , v_{h,\Gamma_\delta} \rangle_{\Gamma_\delta}\big|}{ \| v_{h,\Gamma_\delta}\|_{L^2(\Gamma_\delta)}} .
\end{align*}
By   \eqref{eq:imp-cont}, 
\begin{equation}\label{eq:imp-map-eq}
\left\langle \cI g,  v_\delta \right\rangle_{\Gamma_\delta} = a_{\Omega_\delta}(u, v_\delta) - \alpha(u, \cR_\delta^\top v_\delta) \quad \tfa \  v_\delta \in H^1(\Omega_\delta).
\end{equation}
with $\alpha$ defined by \eqref{eq:alpha}.  
Now we use \eqref{eq:imp-map-eq} with $v_\delta = \wv_{h,\delta} $,  and combine it with
\eqref{eq:imp-map-discrete} to obtain  
\begin{align}\label{eq:L661}
 \big\langle \cI  g -\cI^h g , v_{h,\Gamma_\delta} \big\rangle_{\Gamma_\delta} = a_{\Omega_\delta} (u - u_h, \wv_{h,\delta}) - \alpha\left(u,\cR_\delta^\top \wv_{h,\delta} \right) +a(u_h,\cR_{h, \delta}^\top \wv_{h,\delta}).
\end{align}

The proof is completed by showing that   the last two terms in \eqref{eq:L661} vanish. For the first term we recall that $u$ is Helmholtz-harmonic, and so, by \eqref{eq:alpha}, 
$$ \alpha(u,\cR^\top _\delta \wv_{h,\delta} ) = \big\langle \partial u /\partial n - \ri ku   , \cR_\delta^\top \wv_{h,\delta} \big\rangle_{\partial \Omega};  $$
this last expression vanishes because $\partial u /\partial n - \ri k u = 0 $ on $\partial \Omega \backslash \Gamma^-$ and $\cR_\delta^\top \wv_{h,\delta} = 0 $ on $\Gamma^-$.  Also, by \eqref{eq:discrete-imp-map-sol},
$$a(u_h , \cR_{h,\delta}^\top \wv_{h,\delta} ) = \langle g, \cR_{h,\delta}^\top \wv_{h,\delta} \rangle_{\Gamma^-} = 0 ,$$ thus completing the proof. 
\end{proof}

\bpf[Proof of Lemma \ref{th:65}]
We bound the two terms on the right-hand side of \eqref{eq:strang}. For the first term, by 
  \cite[Theorem 6.1]{DuSc:80} and the standard scaling argument
$$
\inf_{v_{h,\Gamma_\delta} \in {V}^h(\Gamma_\delta)} \|\cI g - v_{h,\Gamma_\delta} \|_{L^2(\Gamma_\delta)}\ \lesssim \  h^{1/2}\|\cI g\|_{H^{1/2}(\Gamma_\delta)}.
$$
Then, with $u$ defined by \eqref{eq:imp-map-sol}, 
\beqs  \Vert \cI g \Vert_{H^{1/2}(\Gamma_\delta)}  =  \left\Vert
 \left(\frac{\partial u}{ \partial x_1}  - \ri k\right) u  \right \Vert_{H^{1/2}(\Gamma_\delta)} \lesssim \ \Vert u \Vert_{H^2(\Omega\backslash \overline{\Omega_\delta})} + k \Vert u \Vert_{H^1(\Omega)} \ \lesssim  \ k\Vert g \Vert_{L^2(\Gamma^-)}, 
 \eeqs
where we used the trace theorem and the bounds  \eqref{eq:reg-1.5} and \eqref{eq:reg-2}, the latter with $D=\Omega \backslash \overline{\Omega_\delta}$. 
Combining the last two displayed estimates gives the first term on the right-hand side of \eqref{eq:th65}.

For the second term on the right-hand side of \eqref{eq:strang}, we recall  that    $
    \wv_{h,\delta}$ is supported in  $\Omega_\delta'$.   
Using   the boundedness of the sesquilinear form $a_{\Omega_\delta}$  (see, e.g.,    \cite[Lemma 2.4(i)]{GrSpZo:20}) we  obtain
\begin{equation}\label{eq:a1local}
  \vert a_{\Omega_\delta}(u-u_h,\wv_{h,\delta} )\vert  \ \lesssim \ \|u-u_h\|_{1,k, \Omega_\delta' } \|  \wv_{h,\delta}  \|_{1,k, \Omega_\delta' }.
\end{equation}
Now, using a standard inverse estimate elementwise on $\Omega_\delta'$ and the local quasiuniformity \eqref{eq:localqu}, we obtain 
\begin{align*} 
  \|\wv_{h,\delta}  \|_{1,k, \Omega_\delta'}^2 & = \sum_{\stackrel{\tau \in T^h}{\tau \subset \Omega_\delta'}}
     \left( \int_\tau \vert \nabla \wv_{h,\delta} \vert^2 + k^2 \vert \wv_{h,\delta}\vert^2 \right)  \lesssim (h^{-2} + k^2) \sum_{\stackrel{\tau \in T^h}{\tau \subset \Omega_\delta'}}
     \int_\tau  \vert \wv_{h,\delta}\vert^2   \\
  & \sim (h^{-2} + k^2) \sum_{\tau \subset \Omega_\delta'} h_\tau^2 \sum_{x_j \in \tau} \vert (\wv_{h,\delta}) (x_j)\vert^2 \quad \quad ( x_j \ \text{are the nodes of the mesh}) \\
  & \sim h (h^{-2} + k^2) \sum_{\tau \subset \Omega_\delta'}  \int_{\Gamma_\delta\cap \overline{\tau}} \vert v_{h,\Gamma_\delta}\vert^2 \quad \quad (\text{since} \ \wv_{h,\delta} \ \text{vanishes at all nodes not on} \ \Gamma_\delta)\\
&  = \  h^{-1}(1 + (hk)^2)  \| v_{h, \Gamma_\delta}\|_{L^2(\Gamma_\delta)}^2.
\end{align*}
Inserting this bound into \eqref{eq:a1local} and using the resulting bound in the second term on the right-hand side of \eqref{eq:strang},
we obtain the second term on the right-hand side of \eqref{eq:th65}; we have therefore proved the result \eqref{eq:th65}.
\epf

\subsection{Proof of Theorem \ref{thm:main_weighted} via weighted error analysis}  \label{subsec:weighted}

\begin{definition}\label{def:omega} 
With $\Omega$ as shown in Figure \ref{fig:hat} and given $\alpha\in [0,1/2]$, let
$\omega \in C^{\infty}(\overline{\Omega})$ be defined by 
    \begin{align} \label{eq:defomega} \omega(\bx)=\widetilde{\omega}(x_1), \quad \bx = (x_1,x_2) \in \overline{\Omega}, \end{align}  
where 
$\widetilde{\omega}\in C^\infty[0,1]$ satisfies    
\begin{equation}\label{eq:weight-func}
 \widetilde{\omega}(x_1)  =\left\{
\begin{aligned}
h^{\alpha},\quad&  x_1 \in [0,\delta - k^{-1}],
\\1, \quad &   x_1 \in [\delta - h, L],
\end{aligned} \right.
\end{equation}
$0 \leq \widetilde{\omega}(x_1) \leq 1 $ for all $x_1 \in [0,L]$, and 
and \begin{align} \label{eq:derivs}  
| \widetilde{\omega}^{(r)}(x_1) | \ \lesssim\  k^r \quad \tfa \ x_1 \in [0,L] \quad \text{and all}\   r \geq 1. 
    \end{align}
\end{definition} 
Note that if $\alpha =0$, then  $\omega \equiv 1$.  

\begin{proposition}[Properties of the weight function $\omega$]
   \label{prop:simple}
  \begin{align}
    \mbox{\hspace{-1cm}} (i) & \quad w \equiv 1 \quad \text{on}\   \Omega_\delta' \quad \text{and} \quad w \equiv h^\alpha \quad \text{on an} \  \mathcal{O}(1)\ \text{neighbourhood of} \ \Gamma^-   ;\label{eq:weight-func-local}\\    
    (ii) &  \quad \Vert \partial^\gamma\omega \Vert_{L^\infty(\Omega)} \ \lesssim \ k^{\vert \gamma \vert}; 
     \label{eq:drw}    \\
  (iii) & \quad  
  \max_{\tau \in {T}^h} \left(\max_{\bx \in \tau} \omega(\bx)/\min_{\bx \in \tau} \omega(\bx)\right) \le\  1+ h^{1-\alpha} k;  \label{eq:taylor} 
  \end{align}
\end{proposition}
\begin{proof}
  Parts (i), (ii) follow from \eqref{eq:weight-func-local} and  \eqref{eq:derivs}. Part (iii) follows from the estimate  
  $$ \max_{\tau \in {T}^h} \left(\sup_{\bx \in \tau} \omega(\bx)/\inf_{\bx \in \tau} \omega(\bx)\right) \le  \ \max_{\tau \in {T}^h} \sup_{\bx_0\in \tau} \sup_{\bx, \tilde{\bx}\in\tau} \frac{ \omega(\bx_0) + |\nabla\omega(\tilde{\bx})| |\bx - \bx_0| }{\omega(\bx_0)}, $$ which is obtained using Taylor's theorem.   \end{proof}

Let
\beq
\label{eq:Omega0}\mathring{\Omega} := \bigcup \big\{ \tau \in T^h: \omega \not\equiv h^{\alpha} \text{ on } \tau \big\},
\eeq
and note that $\mathring{\Omega}$ therefore contains $\supp(\omega- h^\alpha)$.
With a slight abuse of the notation \eqref{eq:Helmen}, we define the weighted norm, for any $u\in H^1(\Omega)$ and
$\beta \in \{ \pm 1\},$
\begin{equation}\label{eq:weighted-norm}
\|u\|_{1,k,\omega^{\beta}}^2 := k^2\|\omega^\beta u\|_{L^2(\Omega)}^2 +  \|\omega^\beta \nabla u\|_{L^2(\Omega)}^2.
\end{equation}

\paragraph{Outline of the proof of Theorem \ref{thm:main_weighted}}

The proof consists of combining the inequality
\begin{align}\label{eq:local-to-weighted}
\|u-u_h\|_{1,k,\Omega_\delta'} \le \|u - u_h\|_{1,k,\omega},
\end{align}
which follows from the definition of $\|\cdot\|_{1,k,\omega}$ and Proposition \ref{prop:simple}, with the following two lemmas.

\ble\label{lem:key}
Suppose Assumption \eqref{ass:mesh} and \eqref{eq:hkone} hold.   If $h^{2-\alpha}k^3$ is sufficiently small, then
\begin{align}
\Vert u - u_h \Vert_{1,k,\omega} 
\ \lesssim\ & \big(1+h^{1-\alpha}k^2 \big) \|u-v_h\|_{1,k,\omega} +  k\big(1+ h^{1-\alpha}k\big)^2  \|u-v_h\|_{L^2(\mathring{\Omega})}\nonumber\\
 &+ hk^2\big(1+ h^{1-\alpha}k\big)^2    \|u-v_h\|_{1,k} \quad \tfa  v_h \in {V}.
\label{eq:keybound2}\end{align}
\ele

\begin{lemma}\label{lm:weighted-intp}
  Let $u\in U_0(\Omega)$ be the solution of  \eqref{eq:imp-map-sol}, let $\alpha \in [0,1/2]$,
  and let $I_h$ be the Scott-Zhang interpolant \cite{ScZh:90} in the space $V$. Then, for any  $\beta > 1/2$,
\begin{equation}\label{eq:intp-error}
\begin{aligned}
\| u-I_h u \|_{1,k,\omega} &\lesssim (h^{\alpha+1/2}k^\beta + hk) \|g\|_{L^2(\Gamma^-)},\\
\|  u-I_h u\|_{L^2(\mathring{\Omega})} & \lesssim h^2k \|g\|_{L^2(\Gamma^-)},   \\
\|  u-I_h u\|_{1,k}&  \lesssim (h^{1/2}k^{\beta} + hk) \|g\|_{L^2(\Gamma^-)} . 
\end{aligned}
\end{equation}
\end{lemma}

\paragraph{Proof of Theorem \ref{thm:main_weighted} assuming Lemmas \ref{lem:key} and \ref{lm:weighted-intp}.}
Combining \eqref{eq:local-to-weighted} and \eqref{eq:keybound2}, and then choosing  $v_h = I_h u$ and using Lemma \ref{lm:weighted-intp}, we obtain
  \begin{align*} 
   \Vert u - u_h \Vert_{1,k,\Omega_\delta'} & \lesssim \Big[ \big(1 + h^{1-\alpha}k^2\big) h^{\alpha+1/2}k
+  k\big(1+ h^{1-\alpha}k\big)^2 h^2k \\
 & \hspace{4.85cm}+  hk^2\big(1+ h^{1-\alpha}k\big)^2 h^{1/2}k
    \Big]
    \Vert g \Vert_{L^2(\Gamma^-)}\\
    & \lesssim \Big[h^{\alpha+1/2}k +h^{3/2}k^3  + \left(h^2k^2 + h^{3/2}k^3\right) \big(1+ h^{1-\alpha}k\big)^2
    \Big]
    \Vert g \Vert_{L^2(\Gamma^-)}.
    \end{align*} 
  Since $h^2k^2 \le h^{3/2}k^3 \le h^{2-\alpha} k^3$, the result follows if we can show that $h^{1-\alpha}k\lesssim 1$.
However, by the hypothesis of the theorem, 
\beqs
h\leq C k^{-3/(2-\alpha)} \quad\text{ so that }\quad h^{1-\alpha}k \leq C k^{1- 3(1-\alpha)/(2-\alpha)},
\eeqs
which is $\lesssim 1$ since $\alpha\leq 1/2$. 
It therefore remains to prove Lemmas  \ref{lem:key} and \ref{lm:weighted-intp}. 

\paragraph{Proof of Lemma \ref{lm:weighted-intp}.}
  
We only prove the first inequality in \eqref{eq:intp-error};  the others follow similarly.
Using the property of the weight function in \eqref{eq:weight-func} and 
applying \cite[Theorem 4.1]{ScZh:90} with $p=2, m=1$, $\ell=3/2$, and then with $\ell=2$,
we have, for all $\beta > 1/2$,
$$
\begin{aligned}
\|\omega \nabla (u - I_h u)\|_{L^2(\Omega)}&\lesssim \|h^\alpha \nabla (u - I_h u)\|_{L^2(\Omega\backslash \mathring{\Omega})} + \| \nabla (u - I_h u)\|_{L^2(\mathring{\Omega})}\\
&\lesssim  h^{\alpha+1/2}\|  u\|_{H^{3/2}(\Omega\backslash \mathring{\Omega})} + h\| u \|_{H^2(\mathring{\Omega})} \le  (h^{\alpha + 1/2} k^\beta+hk) \Vert g \Vert_{L^2(\Gamma^-)}, 
\end{aligned}
$$
where we used the bounds \eqref{eq:reg-1.5} and \eqref{eq:reg-2}, the latter with $D=\mathring{\Omega}$. Similarly, using \cite[Theorem 4.1]{ScZh:90} with $p=2, m=0$, $\ell=3/2$, and then with $\ell=2$, we find that, for all $\beta > 1/2$, 
$$
\begin{aligned}
\|\omega (u - I_h u)\|_{L^2(\Omega)}&\lesssim \|h^\alpha (u - I_h u)\|_{L^2(\Omega\backslash \mathring{\Omega})} + \| (u - I_h u)\|_{L^2(\mathring{\Omega})}
\le  (h^{\alpha + 3/2} k^\beta +h^2k) \Vert g \Vert_{L^2(\Gamma^-)}.
\end{aligned}
$$ 
Combining these last two estimates and using \eqref{eq:hkone} yields the first inequality in \eqref{eq:intp-error}.

\paragraph{Proof of Lemma \ref{lem:key}.} 
This proof  requires several auxiliary results. First
  in Lemma \ref{lm:trace} we prove a trace inequality in the non-standard weighted  setting and use it to obtain the boundedness of the sesquilinear form $a$ in a weighted  context (Corollary \ref{lm:weighted-bounded}).   Then, using a certain   elliptic-projection operator (Definition \ref{def:ellproj}),  we prove in Lemma \ref{lm:weighted-L2} that  $u - u_h$ has a better estimate in the weighted $L^2 $ norm than in the norm $\Vert \cdot \Vert_{1,k, \omega}$
(see Remark \ref{rem:ep} below for discussion on how this result is related to results obtained using this ``elliptic-projection" argument, which was first introduced in the Helmholtz context in \cite{FeWu:09, FeWu:11}).
From Lemma \ref{lm:weighted-L2} we then prove Corollary \ref{co:weighted-L2-total} and Lemma \ref{lm:weighted-H1}, which together prove Lemma    \ref{lem:key}.  
 
\begin{lemma}\label{lm:trace}
Suppose   \eqref{eq:hkone} holds, and let $\beta=\pm1$. Then, for all  $\epsilon \in (0,1]$,
\begin{equation}\label{eq:trace}
k\|\omega^{\beta} u\|_{L^2(\partial \Omega)}^2 \lesssim \epsilon^{-1}k^2\|\omega^\beta u\|_{L^2(\Omega)}^2 + \epsilon \|\omega^\beta \nabla u\|_{L^2(\Omega)}^2  \quad \tfa \ u \in H^1(\Omega),  
\end{equation}
 where the omitted constant is   
   independent of $\epsilon$ (but may depend on $k_0$). 
\end{lemma}
\begin{proof}
  By density, it is sufficient to prove the estimates for all $u \in C^1(\overline{\Omega})$.
In the proof we make use of the standard  inequality
$2ab \leq \epsilon a^2 + \epsilon^{-1}{b^2}$, valid for all  $a,b$ and $\epsilon>0$.

Noting that $\omega$ is constant on $\widetilde{\Omega} := [0,\delta - k_0^{-1}] \times [0,1]$ and using the multiplicative trace inequality, we have
  \begin{align} \Vert \omega^\beta u \Vert_{L^2(\Gamma^-)}  \lesssim  \Vert \omega^\beta u \Vert_{L^2(\widetilde{\Omega})}^{1/2}
  \Vert \nabla(\omega^\beta u) \Vert_{L^2(\widetilde{\Omega})}^{1/2}
   &=  \Vert \omega^\beta u \Vert_{L^2(\widetilde{\Omega})}^{1/2}
                                                              \Vert \omega^\beta \nabla u \Vert_{L^2(\widetilde{\Omega})}^{1/2}\nonumber \\ & \leq {k}{\epsilon^{-1}}  \Vert \omega^\beta u \Vert_{L^2({\Omega})} + {\epsilon}k^{-1} 
\Vert \omega^\beta \nabla u \Vert_{L^2({\Omega})}     
  \label{eq:trace-pf0} \end{align} 
the same estimate holds for $\Vert \omega^\beta u \Vert_{L^2(\Gamma^+)}$, by an analogous argument.

Let $ \boldmu(\bx) = (0, 2 x_2 -1)$ and observe that
    \begin{align} \label{eq:properties} \nabla \omega \cdot \boldmu  = 0\ \text{on} \ \Omega,   \quad \boldmu\cdot\boldnu = 1 \ \text{on}\  \partial \Omega \backslash (\Gamma^- \cup \Gamma^+ ) \quad \text{and} \ \boldmu.\boldnu = 0 \quad \text{on}\   \Gamma^- \cup \Gamma^+,   \end{align}  where $\boldnu$ is the  outward normal on  $\partial \Omega$.  By the divergence theorem,
\beq\label{eq:trace-pf1}
\|\omega^\beta u\|_{L^2(\partial \Omega\backslash (\Gamma^-\cup\Gamma^+))}^2 =  \int_{\partial \Omega}  \omega^{2\beta} \vert u\vert ^2\mathbf{\boldmu}\cdot \mathbf{\boldnu}  = \int_{\Omega}  \nabla\cdot( \omega^{2\beta} \vert u\vert ^2\mathbf{\boldmu}).
\eeq

Since $u, \omega\in C^1(\overline\Omega)$, by the first equation in \eqref{eq:properties},
$$ \nabla \cdot (\omega^{2 \beta}  \vert u \vert^2 \boldmu) = \omega^{2 \beta} \vert u \vert^2 \nabla \cdot \boldmu + 2 \omega^{2 \beta} \Re (\overline{u} \nabla u)\cdot\boldmu;  $$ then, by  the Cauchy-Schwarz inequality,
we have
\begin{equation}\label{eq:trace-pf2}
\begin{aligned}
  \left|\int_{\Omega}  \nabla\cdot( \omega^{2\beta} \vert u \vert ^2\mathbf{\boldmu}) \right|
&\le  \Vert \nabla\cdot \mathbf{\boldmu}\Vert _{L^\infty(\Omega)} \|\omega^\beta u\|_{L^2(\Omega)}^2 + 2\Vert  \mathbf{\boldmu}\Vert _{L^\infty(\Omega)} \|\omega^\beta u\|_{L^2(\Omega)}  \|\omega^\beta \nabla u\|_{L^2(\Omega)}\\
&\le 2 (1 +\epsilon^{-1}k)  \|\omega^\beta u\|_{L^2(\Omega)}^2 + \epsilon k^{-1}  \|\omega^\beta \nabla u\|_{L^2(\Omega)}^2;
\end{aligned}
\end{equation}
the result then follows from \eqref{eq:trace-pf0}, \eqref{eq:trace-pf1} and \eqref{eq:trace-pf2} since we are assuming $k\ge k_0$.
\end{proof}

Lemma \ref{lm:trace} allows us to prove the following boundedness estimate for the sesquilinear form $a(\cdot,\cdot)$ in the weighted setting.
\begin{corollary}\label{lm:weighted-bounded} Given $k_0>0$
  there exists a constant $\Ccont$ such that,
for all $k\geq k_0$,
 \begin{equation*}
\vert a(u,v)\vert  \le \Ccont \|u\|_{1,k,\omega} \|v\|_{1,k,\omega^{-1}} \quad \tfa\  u,v \in H^1(\Omega).
\end{equation*}
\end{corollary}
\begin{proof}
By Cauchy-Schwarz inequality, 
\begin{align}
\vert a(u,v) \vert &=\left\vert (\omega\nabla u , \omega^{-1}\nabla v) - k^2 (\omega u , \omega^{-1}v) - \ri k  \langle  \omega u,\omega^{-1} v\rangle_{\partial \Omega} \right\vert \nonumber  \\
&\lesssim \|\omega\nabla u\|_{L^2(\Omega)}   \|\omega^{-1}\nabla v\|_{L^2(\Omega)} +k^2 \|\omega  u\|_{L^2(\Omega)}   \|\omega^{-1}  v\|_{L^2(\Omega)} + k \| \omega u\|_{L^2(\partial \Omega)} \,  \|\omega^{-1} v\|_{L^2(\partial \Omega)}\label{eq:boundary}\\
&\lesssim \left( \|\omega\nabla u\|_{L^2(\Omega)} ^2 + k^2 \|\omega  u\|_{L^2(\Omega)}   \right)^{1/2}   \left(\|\omega^{-1}\nabla v\|_{L^2(\Omega)}^2 + k^2 \|\omega^{-1}  v\|_{L^2(\Omega)}^2\right)^{1/2}, 
                                                                                                                                       \nonumber \end{align}
  where the last term in \eqref{eq:boundary} is estimated using \eqref{eq:trace} with $\epsilon = 1$. 
\end{proof}

Define the sesquilinear form $a_\star(\cdot,\cdot)$ by 
\beqs
a_\star(u,v) :=( \nabla u, \nabla v)_\Omega - \ri k \langle u, v\rangle_{\partial \Omega}.
\eeqs
Then $a_\star(\cdot,\cdot)$ is continuous and coercive on $H^1(\Omega)$ and 
$
\N{v}_\star := \sqrt{a_\star(v,v)}
$
is a norm on $H^1(\Omega)$ with 
\beq\label{eq:normequiv}
\N{v}_\star \sim \N{\nabla v }_{L^2(\Omega)}+  k^{1/2} \N{v}_{L^2(\partial \Omega)};
\eeq
see, e.g., \cite[Lemma 5.3]{ChNi:18}, \cite[Lemma 7.1]{LaSpWu:19a}.

\begin{definition}[Elliptic projection $\cP_h$]
\label{def:ellproj}Given $u\in H^1(\Omega)$, define $\cP_h u \in {V}$ by 
\beqs
a_\star(v_h, \cP_h u) = a_\star(v_h,u)\quad\tfa v_h\in V.
\eeqs
\end{definition}
By the Lax--Milgram theorem,  $\cP_h$ is well-defined  and we have the  Galerkin orthogonality
\beq\label{eq:GO2}
a_\star(v_h ,u-\cP_h u) = 0 \quad\tfa v_h\in V.
\eeq

\ble[Approximation properties of $\cP_h$] For all $u\in H^1(\Omega)$,
\begin{align}
\label{eq:ep1}
\N{u-\cP_h u}_\star \lesssim \min_{v_h\in {V}} \N{u-v_h}_{1,k} \quad\tand\quad
\N{u-\cP_h u}_{L^2(\Omega)} \lesssim h\N{u-\cP_h u}_\star.
\end{align}
\ele
\bpf[References for the proof]
This follows from \cite[\S5.5]{ChNi:18} or \cite[Lemma 7.4]{LaSpWu:19a}, using the regularity result of \cite{ChNiTo:20}.
\epf

The next lemma uses the following  regularity result for the adjoint problem with zero impedance data. Let $\phi\in H^1(\Omega)$ be the solution of $a(v,\phi) = (v,f)$ for all $v\in H^1(\Omega)$. 
Then, by \cite[Lemma 2.12]{GaGrSp:15}, 
\begin{equation}\label{eq:adj-reg}
\|\phi\|_{H^2(\Omega)} \lesssim k\|f\|_{L^2(\Omega)}.
\end{equation}

\begin{lemma}\label{lm:weighted-L2}
If $\alpha \in [0,1/2]$ and $h^{2-\alpha}k^3$ is sufficiently small, then, 
with $\partial^r$ any  partial derivative  of order $r\geq 0$,
\begin{equation}\label{eq:L2weighted-error1}
\|(\partial^r \omega )(u-u_h) \|_{L^2(\Omega)} \ \lesssim\  h^{1-\alpha}k^{r+1}    \|u-v_h\|_{1,k,\omega}\quad\tfa
v_h\in {V}.
\end{equation}
 \end{lemma}
 \begin{proof}   
Set $e := u-u_h$ and  let $\phi\in H^1(\Omega)$ be the solution of the adjoint problem   
\begin{equation}\label{eq:duality}
a(v,\phi) = \left(v,(\partial^r \omega)^2 e \right) \quad \tfa \  v\in H^1(\Omega).
\end{equation}
By \eqref{eq:adj-reg}
and \eqref{eq:drw}, 
\begin{equation}\label{eq:H2helm}
\|\phi\|_{H^2(\Omega)} \lesssim k\|(\partial^r \omega)^2 e\|_{L^2(\Omega)}  \lesssim k^{r+1} \|(\partial^r \omega) e\|_{L^2(\Omega)}.  
\end{equation}
The motivation for considering this particular adjoint problem is the following.
By \eqref{eq:duality}, Galerkin orthogonality for both $a(\cdot,\cdot)$ and  $a_\star(\cdot,\cdot)$ \eqref{eq:GO2}, for any $v_h\in \mathrm{V}$,
\begin{equation}\label{eq:dual-orth}
\|(\partial^r \omega) e\|_{L^2(\Omega)}^2 = a(e,\phi) ={a}(e, \phi - \cP_h \phi) =a_\star(u - v_h, \phi - \cP_h \phi) - k^2(e,\phi - \cP_h \phi).
\end{equation}
We see that to obtain \eqref{eq:L2weighted-error1} from \eqref{eq:dual-orth}, we need to bound $ \phi - \cP_h \phi$.
By \eqref{eq:normequiv}, \eqref{eq:ep1}, \eqref{eq:hkone}  and \eqref{eq:H2helm},
\begin{align}\label{eq:ellipitc-pro-appro0}
   \|\nabla(\phi - \cP_h \phi)\|_{L^2(\Omega)} +k^{1/2} \|\phi -\cP_h \phi\|_{L^2(\partial \Omega)} &\lesssim \min_{v_h \in V} \Vert \phi - v_h\Vert_{1,k} \nonumber \\
  &\lesssim h(1+hk) \|\phi\|_{H^2(\Omega)}\lesssim h k^{r+1} \|(\partial^r \omega) e\|_{L^2(\Omega)},
  \end{align}
Arguing similarly, but starting with the second inequality in \eqref{eq:ep1}, we find that
\begin{align} \|\phi - \cP_h \phi\|_{L^2(\Omega)} &\lesssim h^2 \|\phi\|_{H^2(\Omega)}\lesssim h^2 k^{r+1} \|(\partial^r \omega) e\|_{L^2(\Omega)}. \label{eq:ellipitc-pro-appro} 
\end{align}
For the second term of the right-hand side of \eqref{eq:dual-orth}, Definition \ref{def:omega} and \eqref{eq:ellipitc-pro-appro} imply that
\begin{align}
\left| k^2(e,\phi - \cP_h \phi)\right| = k^2\left| \left(\omega e,\omega^{-1}(\phi - \cP_h \phi) \right)\right| &\leq  k^2 \|\omega e\|_{L^2(\Omega)} \|\omega^{-1}(\phi - \cP_h \phi) \|_{L^2(\Omega)}\nonumber  \\ \nonumber
&\leq  k^2 \|\omega e\|_{L^2(\Omega)} h^{-\alpha} \|\phi - \cP_h \phi \|_{L^2(\Omega)} \\ &\lesssim  h^{2-\alpha }k^{r+3} \|\omega e\|_{L^2(\Omega)} \|(\partial^r\omega) e\|_{L^2(\Omega)}. 
\label{eq:dual-l2}\end{align}
For the first term of the right-hand side of \eqref{eq:dual-orth}, we write
$$a_\star(u-v_h,\phi - \cP_h \phi) = (\omega \nabla(u-v_h), \omega^{-1} \nabla (\phi -\cP_h \phi)) -\ri k \langle \omega (u-v_h), \omega^{-1} (\phi-\cP_h \phi)  \rangle_{\partial \Omega}; $$
using the Cauchy-Schwarz inequality (as in the proof of Corollary \ref{lm:weighted-bounded}) we obtain
\begin{align}
\left|a_\star(u-v_h,\phi - \cP_h \phi)\right| 
&\lesssim   \|u-v_h\|_{1,k,\omega} \|\phi -\cP_h \phi\|_{1,k,\omega^{-1}}
\le  \|u-v_h\|_{1,k,\omega} h^{-\alpha}\|\phi -\cP_h \phi\|_{1,k}\nonumber \\
&\mbox{\hspace{1cm}} \lesssim h^{1-\alpha}k^{r+1}  \|u-v_h\|_{1,k,\omega} \|(\partial^r \omega) e\|_{L^2(\Omega)},
\label{eq:dual-h1}
\end{align}
where in the last step we used \eqref{eq:ellipitc-pro-appro0}, \eqref{eq:ellipitc-pro-appro} and \eqref{eq:hkone}.   
Combining \eqref{eq:dual-orth}, \eqref{eq:dual-l2} and \eqref{eq:dual-h1}, we find that
\begin{align} \label{eq:614}
\|(\partial^r\omega) e\|_{L^2(\Omega)}\lesssim h^{1-\alpha}k^{r+1}  \|u-v_h\|_{1,k,\omega} + h^{2-\alpha }k^{r+3} \|\omega e\|_{L^2(\Omega)} .
\end{align} 
Since this holds for all $r \geq 0$ we can choose    $r=0$ and  use the assumption that $h^{2-\alpha}k^3$ is sufficiently small to obtain that
$$
\|\omega e\|_{L^2(\Omega)}\lesssim h^{1-\alpha}k  \|u-v_h\|_{1,k,\omega}.
$$
Inserting this into the right-hand side of \eqref{eq:614} yields 
$$
\begin{aligned}
  \|(\nabla_r\omega) e\|_{L^2(\Omega)}
&\lesssim h^{1-\alpha}k^{r+1}  \|u-v_h\|_{1,k,\omega} + h^{2-\alpha }k^{r+3} h^{1-\alpha}k \|u-v_h\|_{1,k,\omega} \\
&=h^{1-\alpha}k^{r+1} (1 + h^{2-\alpha}k^3) \|u-v_h\|_{1,k,\omega};
\end{aligned}
$$
thus, the result \eqref{eq:L2weighted-error1} follows from the fact that  $h^{2-\alpha}k^3$ is sufficiently small.
\end{proof}
The following corollary is obtained by putting $r = 0$ in 
  Lemma \ref{lm:weighted-L2} and bounds the $L^2$ norm on the left-hand side of \eqref{eq:keybound2}. 
\begin{corollary}\label{co:weighted-L2-total}
If $h^{2-\alpha}k^3$ is sufficiently small, then
\begin{equation}\label{eq:L2part}
 \|\omega(u-u_h)\|_{L^2({\Omega})}\lesssim h^{1-\alpha} k\|u-v_h\|_{1,k,\omega} \quad \tfa v_h\in V.
\end{equation}
\end{corollary}

\bre[Link with other work using the ``elliptic-projection argument'']\label{rem:ep}
Recall that the ``Schatz argument" proves, for the FEM applied to \eqref{eq:Helm}, under the assumption of $H^2$ regularity, the Aubin-Nitsche type bound that
\beqs
 \|u-u_h\|_{L^2({\Omega})}\lesssim hk\|u-u_h\|_{1,k};
\eeqs
see \cite{Sc:74,Sa:06}. The ``elliptic projection'' argument applied to this set up proves the stronger bound that, if $h^{2}k^3$ is sufficiently small, then, for all $v_h\in V$,
\beq\label{eq:ep_classic}
 \|u-u_h\|_{L^2({\Omega})}\lesssim hk\|u-v_h\|_{1,k}; 
\eeq
see, e.g., \cite[Theorem 5.2]{ChNi:18}, \cite[Equation 5.2]{LaSpWu:19a}, with this argument relying on the regularity result of \cite{ChNiTo:20}. This argument was first introduced in the setting of discontinuous Galerkin methods by \cite{FeWu:09,FeWu:11}, with the analogue of \eqref{eq:ep_classic} appearing in \cite[Lemma 5.2]{FeWu:09}, \cite[Lemma 4.3]{FeWu:11}, and \cite[Lemma 4.2]{Wu:14}; see the literature review in, e.g., \cite[\S2.3]{Pe:20}.

Observe that our weighted estimate \eqref{eq:L2part} generalises \eqref{eq:ep_classic}, and reduces to the latter when $\alpha=0$ (and hence $\omega\equiv 1$ by \eqref{eq:weight-func}).
\ere

To complete the proof of Lemma \ref{lm:weighted-L2} we need to bound the $H^1$ semi-norm on the left-hand side of \eqref{eq:keybound2}. This bound is obtained in the following lemma.   
\begin{lemma}\label{lm:weighted-H1}
    Let $\alpha \in [0,1/2]$. 
    If $h^{2-\alpha}k^3$ is sufficiently small, then
 for any $v_h\in {V},$
\begin{align}\nonumber
\|\omega\nabla(u - u_h)\|_{L^2(\Omega)} \lesssim& \big(1+h^{1-\alpha}k^2  \big) \|u-v_h\|_{1,k,\omega} + k  \big(1+ h^{1-\alpha}k\big)^2\|u-v_h\|_{L^2(\mathring{\Omega})}\\
& \quad \quad + hk^2 \big(1+ h^{1-\alpha}k\big)^2 \|u-v_h\|_{1,k}.
\label{eq:H1weighted-error}
\end{align}
\end{lemma}

\begin{proof}
  Set $e = u - u _h$. By the definition of $a(\cdot,\cdot)$,
\begin{equation}\label{eq:H1-terms}
\|\omega\nabla e\|_{L^2(\Omega)}^2  =
a(e,\omega^2 e) -(\nabla e , 2\omega (\nabla \omega) e) + k^2(e,\omega^2 e) + \ri k\langle e, \omega^2 e\rangle;
\end{equation}
we now bound each of the four terms on the right-hand side.

In what follows, $v_h$ is an arbitrary element of $V$ and $\epsilon \in (0,1]$ is arbitrary.
For the second term on the right-hand side of \eqref{eq:H1-terms}, using Cauchy-Schwarz inequality and \eqref{eq:L2weighted-error1}, we have
\begin{equation}\label{eq:H1-terms-2}
\begin{aligned}
\left| (\nabla e , 2\omega (\nabla \omega )e)  \right| &\le \epsilon \|\omega \nabla e\|_{L^2(\Omega)}^2 + \epsilon^{-1} \|(\nabla \omega) e\|_{L^2(\Omega)}^2 \\ 
&\lesssim \epsilon \|\omega \nabla e\|_{L^2(\Omega)}^2 + \epsilon^{-1} (h^{1-\alpha}k^2)^2 \|u-v_h\|_{1,k,\omega}^2.
\end{aligned}
\end{equation}
For the third term on the right-hand side of \eqref{eq:H1-terms}, using  \eqref{eq:L2weighted-error1}, we have
\begin{equation}\label{eq:H1-terms-3}
\begin{aligned}
k^2 ( e , \omega^2   e) \ = \ k^2 \Vert \omega e \Vert_{L^2(\Omega)}^2 \   \lesssim \   (h^{1-\alpha}k^2)^2 \|u-v_h\|_{1,k,\omega}^2.
\end{aligned}
\end{equation}
For the fourth term on the right-hand side of \eqref{eq:H1-terms}, using  \eqref{eq:trace} and then \eqref{eq:L2weighted-error1}, we have, 
\begin{equation}\label{eq:H1-terms-4}
\begin{aligned}
 k\langle e, \omega^2 e\rangle_{\partial \Omega}\ = \ k \Vert \omega e \Vert_{L^2(\partial \Omega)}^2  &\lesssim 
 \epsilon \|\omega \nabla e\|_{L^2(\Omega)}^2+
 \epsilon^{-1}k^2\|\omega  e\|_{L^2(\Omega)}^2 \\
 &\lesssim    \epsilon \|\omega \nabla e\|_{L^2(\Omega)}^2 +  \epsilon^{-1}(h^{1-\alpha}k^2)^2  \|u-v_h\|_{1,k,\omega}^2\\
 \end{aligned}
\end{equation}

Therefore, combining \eqref{eq:H1-terms-2}, \eqref{eq:H1-terms-3}, \eqref{eq:H1-terms-4}, and \eqref{eq:H1-terms}, and then choosing the $\epsilon$ parameters in  \eqref{eq:H1-terms-2} and \eqref{eq:H1-terms-4} to be sufficiently small (independent of $k$), we obtain  
\beq\label{eq:takestock1}
\|\omega\nabla e\|_{L^2(\Omega)}^2  \lesssim
|a(e,\omega^2 e) |
+  \big(h^{1-\alpha}k^2\big)^2  \|u-v_h\|_{1,k,\omega}^2.
\eeq

More effort is needed to bound the first term on the right-hand side of \eqref{eq:H1-terms}/\eqref{eq:takestock1}.  For any $v_h \in V$, set 
\begin{align} \label{eq:defpsih} \psi_h := \omega^2( u_h- v_h).\end{align}  
Then, by Galerkin orthogonality,   
\begin{align}
a(e,\omega^2 e) = a(u-u_h,\omega^2 (u-u_h)) \nonumber 
& =a(u-u_h,\omega^2 (u-v_h)) - a(u-u_h,\omega^2 (u_h-v_h))\nonumber \\
& =a(u-u_h,\omega^2 (u-v_h)) - a(u-u_h,\psi_h - \Pi_h\psi_h),
\label{eq:H1-terms-1}\end{align}
where $\Pi_h$ denotes the nodal interpolation operator that maps a continuous function on $\Omega$ into its interpolant in $V$.

We now deal with each of the two  terms in \eqref{eq:H1-terms-1} separately. 
For the first term, by Corollary \ref{lm:weighted-bounded} and  the definition \eqref{eq:weighted-norm} of $\Vert \cdot \Vert_{1,k,\omega}$,
\begin{align}\nonumber
\left| a(u-u_h,\omega^2 (u-v_h))\right| &\ \lesssim \  \|u-u_h\|_{1,k,\omega} \|\omega^2(u-v_h)\|_{1,k,\omega^{-1}}\\
&\le \epsilon\left(\|\omega \nabla e\|_{L^2(\Omega)}^2 + k^2\|\omega e\|_{L^2(\Omega)}^2 \right) + \epsilon^{-1} \|\omega^2(u-v_h)\|_{1,k,\omega^{-1}}^2.\label{eq:combine1}
\end{align}
Then, by \eqref{eq:drw} and \eqref{eq:Omega0},
\begin{align}\nonumber
\|\omega^2(u-v_h)\|_{1,k,\omega^{-1}}^2 &=\|\omega^{-1}\nabla\left(\omega^2 (u-v_h)\right)\|_{L^2(\Omega)}^2 + k^2 \|\omega^{-1}\omega^2(u-v_h)\|_{L^2(\Omega)}^2 \\ \nonumber
&\lesssim \|u-v_h\|_{1,k,\omega}^2 + \|2(\nabla  {\omega}) (u-v_h)\|_{L^2(\Omega)}^2\\
&\lesssim \|u-v_h\|_{1,k,\omega}^2 + k^2 \|u-v_h\|_{L^2(\mathring{\Omega})}^2.\label{eq:combine2}
\end{align}
Thus, combining \eqref{eq:combine1}, \eqref{eq:combine2}, and \eqref{eq:L2part}, we obtain the following estimate for the first term in \eqref{eq:H1-terms-1}:
\begin{align}
  \left| a(u-u_h,\omega^2 (u-v_h))\right| &\lesssim \epsilon\|\omega \nabla e\|_{L^2(\Omega)}^2 + \left(\epsilon(h^{1-\alpha}k^2)^2 + \epsilon^{-1}\right)\|u-v_h\|_{1,k,\omega}^2
     +\epsilon^{-1}k^2 \|u-v_h\|_{L^2(\mathring{\Omega})}^2. \label{eq:H1-terms1-1}
\end{align}
For the second term in  \eqref{eq:H1-terms-1}, by Corollary \ref{lm:weighted-bounded} and  \eqref{eq:L2part}, 
\begin{align}
\left| a(u-u_h,\psi_h -\Pi_h\psi_h )\right| &\lesssim  \|u-u_h\|_{1,k,\omega} \|\psi_h -\Pi_h\psi_h\|_{1,k,\omega^{-1}}\nonumber \\
& \mbox{\hspace{-2cm}} \le \epsilon\left(\|\omega \nabla e\|_{L^2(\Omega)}^2 + k^2\|\omega e\|_{L^2(\Omega)}^2 \right) + \epsilon^{-1} \|\psi_h -\Pi_h\psi_h\|_{1,k,\omega^{-1}}^2 \nonumber \\
& \mbox{\hspace{-2cm}} \le \epsilon \|\omega \nabla e\|_{L^2(\Omega)}^2 + \epsilon (h^{1-\alpha}k^2)^2\|u-v_h\|_{1,k,\omega} ^2+ \epsilon^{-1} \|\psi_h -\Pi_h\psi_h\|_{1,k,\omega^{-1}}^2.
\label{eq:616a}\end{align}
Now, with $h_\tau$ denoting the diameter of $\tau$,    standard element-wise estimates for $\Pi_h$ yield 
\begin{align}\nonumber
\|\psi_h -\Pi_h\psi_h\|_{1,k,\omega^{-1}} ^2 &=\sum_{\tau\in{T}^h} \int_{\tau}\omega^{-2} \left(k^2|\psi_h -\Pi_h\psi_h|^2 + |\nabla (\psi_h -\Pi _h\psi_h)|^2\right)~dx\\ \nonumber
&\le\sum_{\tau \in {T}^h}  \left({\min_{\bx\in \tau}}\, \omega(\bx)\right)^{-2} \int_{\tau} \left(k^2|\psi_h -\Pi_h\psi_h|^2 + |\nabla (\psi_h -\Pi_h\psi_h)|^2\right)~dx\\ 
&\lesssim \sum_{\tau\in {T}^h} {\left(\min_{\bx\in \tau}\omega(\bx)\right)^{-2}}
                                                                                                                                                                               \left((hk)^2+1\right)h_\tau^{2p}|\psi_h|_{H^{p+1}(\tau)}^2~dx    \nonumber \\
  & \lesssim (1 + h^{1-\alpha} k)^2 \sum_{\tau \in T^h} \sum_{\vert \beta \vert = p+1} h_\tau^{2p} \Vert \omega^{-1} (\partial ^\beta \psi_h) \Vert_{L^2(\tau)}^2 ,  
  \label{eq:problem?1}                                               
\end{align}
where in the last step we used \eqref{eq:hkone} and   \eqref{eq:taylor}.

Now, to estimate  \eqref{eq:problem?1}, we  recall the definition \eqref{eq:defpsih} of $\psi_h$ and use the multivariate  Leibnitz rule (with $\beta$ and $\gamma$ denoting multiindices)  to obtain 
\begin{align} \label{eq:B}
  \Vert \omega^{-1} (\partial^\beta \psi_h) \Vert_{L^2(\tau)} & \lesssim  \sum_{0 \leq \gamma\leq \beta} \Vert \omega^{-1} \partial^\gamma(\omega^2) \partial^{\beta - \gamma} (u_h - v_h)\Vert_{L^2(\tau)}\\
  & =   \sum_{r=1}^{p+1} \sum_{\vert \gamma\vert = r } \Vert \omega^{-1} \partial^\gamma(\omega^2) \partial^{\beta - \gamma} (u_h - v_h)\Vert_{L^2(\tau)}, \nonumber
\end{align}
where in the last step we used the fact that   $\vert \beta \vert  = p+1$ and 
$u_h - v_h$ is a  polynomial of degree $p$ on $\tau$,  the term with $\gamma = 0$ in \eqref{eq:B}  vanishes. Also, simple induction shows that
\beqs\Vert \omega^{-1} \partial^\gamma(\omega^2) \Vert_{L^\infty(\Omega)} \lesssim \left\{ \begin{array}{ll} k, & \text{when} \ \vert \gamma \vert = 1,  \\ h^{-\alpha} k^{\vert \gamma \vert}& \text{when} \ \vert \gamma \vert > 1. \end{array} \right. 
\eeqs
Using this and an element-wise inverse estimate for derivatives of polynomials on $\tau$, we obtain (using  again \eqref{eq:hkone})
\begin{align}   \Vert \omega^{-1} (\partial^\beta \psi_h) \Vert_{L^2(\tau)} & \lesssim
                                                                              \left( h_\tau^{-p}k  + h^{-\alpha} \sum_{r=2}^{p+1}   h_\tau^{r-p-1} k^r \right) \Vert u_h - v_h \Vert_{L^2(\tau)}\nonumber \\ & \lesssim
                                                                                h_\tau^{-p} k \left(1   + h^{1-\alpha} k  \right) \Vert u_h - v_h \Vert_{L^2(\tau)}. \label{eq:C}  \end{align} 
                                                                              Also, since the derivatives of $\omega$ vanish outside of $\mathring{\Omega}$,  the upper bound  \eqref{eq:C}  can be replaced by zero for elements $\tau$ that do not intersect $\mathring{\Omega}$. Hence, combining \eqref{eq:C} and  \eqref{eq:problem?1} yields
\begin{align*}                                                                                  
\Vert \psi_h - \Pi_h \psi_h \Vert_{1,k,\omega^{-1}}^2   
                                                        &\lesssim \left(1+h^{1-\alpha }k \right)^4 k^2 \|u_h-v_h\|_{L^2(\mathring{\Omega})}^2, 
\end{align*} 
 Inserting this into \eqref{eq:616a}, we have 
\begin{equation}\label{eq:H1-terms1-2}
\begin{aligned}
\left| a(u-u_h,\psi -\Pi_h\psi )\right| 
  &\lesssim \epsilon\|\omega \nabla e\|_{L^2(\Omega)}^2 
  + \epsilon(h^{1-\alpha}k^2)^2 \|u-v_h\|_{1,k,\omega}^2\\
&  \qquad \quad +  \epsilon^{-1} \left(1+h^{1-\alpha }k\right)^4 k^2  \|u_h-v_h\|_{L^2(\mathring{\Omega})}^2.
\end{aligned}
\end{equation}
Now, by the triangle inequality and  \eqref{eq:L2part} (with $\alpha = 0$ and hence $\omega \equiv 1$),
\beq\label{eq:Friday1}
\|u_h-v_h\|_{L^2(\mathring{\Omega})}\lesssim (hk) \|u-v_h\|_{1,k} + \|u-v_h\|_{L^2(\mathring{\Omega})}.
\eeq
Combining \eqref{eq:H1-terms-1}, \eqref{eq:H1-terms1-1}, \eqref{eq:H1-terms1-2}, and \eqref{eq:Friday1}, 
 we obtain 
\begin{align*}\nonumber
|a(e,\omega^2 e)|\leq &\epsilon \|\omega\nabla e\|^2_{L^2(\Omega)}
+ \Big(\epsilon( h^{1-\alpha}k^2)^2 + \epsilon^{-1} \Big) \|u-v_h\|^2_{1,k,\omega}\\
&\quad+ \epsilon^{-1} \big(1+ h^{1-\alpha}k\big)^4  h^2k^4\|u-v_h\|^2_{1,k}+ \epsilon^{-1}\big(1+ h^{1-\alpha}k\big)^4 k^2\|u-v_h\|^2_{L^2(\mathring{\Omega})}.
\end{align*}
The result \eqref{eq:H1weighted-error} then follows from combining this last inequality with \eqref{eq:takestock1}.
\epf

\bre[Generalising Theorem \ref{thm:main_weighted} to more-general geometries]
\label{rem:more_general_geometries}
Inspecting the proof of Theorem \ref{thm:main_weighted}, 
we see that the key result to generalise is the trace inequality of Lemma \ref{lm:trace}. The proof of this uses the fact that in the region where $\omega$ is non-constant (in a neighbourhood of $\Gamma_\delta$), it is constant in the normal direction -- see the first equation in \eqref{eq:properties}.

Thus an analogous result to Theorem \ref{thm:main_weighted} holds for any convex polygon and interface $\Gamma_\delta$, provided a weight function with analogous properties can be constructed. This is the case, in particular, when the interior interface is separated from \emph{both} the boundary where the non-zero impedance condition is imposed \emph{and} vertices of the polygon.

The requirement that the polygon is convex could be removed, but then one would need to use $H^{3/2+\eps}$ regularity instead of $H^2$ regularity, and this would change the powers of $h$ and $k$ in the final error bound.
\ere

\section{Numerical experiments}\label{sec:numerical}
In this section, we give numerical experiments to validate our theoretical results on the convergence of the impedance-to-impedance maps, and on the  performance of the ORAS preconditioner. Because there are substantial results in other papers \cite{GoGaGrLaSp:21, GoGaGrSp:21, GoGrSp:21},  we are brief; in particular,
\cite{GoGaGrLaSp:21} used the ORAS iteration for small $h$
as an illustration of the theory on convergence of the parallel Schwarz iteration.     
All the experiments were implemented within the FreeFEM++ software \cite{hecht2019freefem++} and were run on the University of Bath's Balena HPC system. All the experiments concern rectangular 2-d  domains,
discretised using nodal elements of degree $p = 1,2,3,4$ on  uniform triangular meshes.

\subsection{Computations of  the impedance-to-impedance maps}

We  illustrate the convergence result Theorem \ref{thm:main}  by computing  impedance-to-impedance maps on the canonical domain in Figure
\ref{fig:hat}. The map that takes left-facing impedance data on $\Gamma^-$ to left-facing impedance data on $\Gamma_\delta$ we call $\cI_{-\rightarrow-}$. The map that takes right-facing impedance data on $\Gamma^+$ to left-facing impedance data on $\Gamma_\delta$ we call $\cI_{+\rightarrow-}$.   The finite element approximations are $\cI^h_{-\rightarrow-}, \cI^h_{+\rightarrow-}$.
With $\cI^h$ denoting either of these maps, \eqref{eq:Th-IhnNorm1} implies that $\cI^h g =\cI^h g_h$ where  $g_h$ is the $L^2$-orthogonal projection of $g$ onto the corresponding finite element space.
Therefore $\cI^h$ acts only on finite-dimensional spaces, and its norm can be computed by solving an appropriate matrix eigenvalue problem - more details are in \cite{GoGaGrLaSp:21}. 

\paragraph{First experiment.}
This experiment verifies  Theorem \ref{thm:main}. Since the exact norm of any map $\cI$
is unknown, we approximate it
by computing $\|\cI^{h_0} \|$ with $h_0 =1/(32k)$. 
Tables \ref{tb:imp-left2right-error}-\ref{tb:imp-left2left-error} give the errors $| \| \cI^h_{t \rightarrow -} \| - \| \cI^{h_0}_{t \rightarrow -} \| |$ for $t=+$ and $t=-$, respectively.
Both tables show that the error is decreasing as $h$ decreases and mostly the rate is faster than $\mathcal{O}(h)$.  In this experiment, the finite-element degree $p=1$ and the quantity $\delta$ in Figure \ref{fig:hat} is chosen to be $1/4$.  
\begin{table}[H]
\setlength\extrarowheight{2pt} 
\centering
\scalebox{0.9}{
\begin{tabularx}{\textwidth}{C|CCCC}
\hline
\hline
$k \backslash h$ & $\frac{1}{2k}$  &  $\frac{1}{4k}$	& $\frac{1}{8k}$  &  $\frac{1}{16k}$			\\ 
\hline 
$10$	&1.3e-2	&4.9e-3	&1.6e-3	&5.0e-4\\
$20$	&6.4e-3	&2.2e-3	&2.0e-3	&8.2e-4\\
$40$	&2.1e-2	&2.3e-3	&1.2e-4	&6.4e-5\\
$80$	&3.5e-2	&8.9e-3	&1.4e-3	&1.4e-4\\
\hline
	\hline
\end{tabularx}}
\caption{Computations of   $ \| \cI^h_{+ \rightarrow -} \| - \| \cI^{h_0}_{+ \rightarrow -} \| |$ with $h_0 =1/(32k).$
}\label{tb:imp-left2right-error}
\end{table}

\begin{table}[H]
\setlength\extrarowheight{2pt} 
\centering
\scalebox{0.9}{
\begin{tabularx}{\textwidth}{C|CCCC}
\hline
\hline
$k \backslash h$ & $\frac{1}{2k}$  &  $\frac{1}{4k}$	& $\frac{1}{8k}$  &  $\frac{1}{16k}$			\\ 
\hline 
$10$	&6.4e-3	&2.4e-3	&1.1e-3	&3.9e-4\\
$20$	&5.6e-3	&2.1e-4	&3.2e-5	&1.2e-5\\
$40$	&5.4e-3	&1.8e-5	&7.4e-6	&1.1e-6\\
$80$	&2.6e-3	&7.5e-6	&5.2e-4	&7.6e-5\\
\hline
	\hline
\end{tabularx}}
\caption{Computations of  $ | \| \cI^h_{-\rightarrow -} \| - \| \cI^{h_0}_{-\rightarrow-} \| |$ with $h_0 =1/(32k).$
}\label{tb:imp-left2left-error}
\end{table}

\paragraph{Second experiment.}
This experiment studies how the norms of these maps vary with $\delta$.  Table
  \ref{tb:imp-left2right}
  gives results for
$\| \cI^h_{+\rightarrow - } \|$ (a right-to left map)
and we observe that it  decreases as $\delta$ increases.
(A rigorous proof of a closely related result is described in \cite[\S4.4.4]{GoGaGrLaSp:21}.
The smallness of this norm is a driver for the power contractivity of $\bcTh$ --  see  \S \ref{subsec:power_recap}.)

 Table   \ref{tb:imp-left2left} gives results for  $\| \cI^h_{- \rightarrow -} \|$ (a  left-to-left impedance map).  This norm remains very close to $1$. 
In this experiment, we fixed $h = 80^{-5/4}$ and used polynomials of degree  $p=2$.

More computations on the impedance-impedance maps can be found in \cite{GoGaGrLaSp:21}. A summary of the results  is as follows.
\begin{enumerate}
\item  The norms of  right-to-left (or left-to-right)  impedance maps can be made
  controllably small by making both the length of the canonical domain and the overlap  long enough (see \cite[Experiment 1]{GoGaGrLaSp:21}).
\item The norm of the composition of such  maps is much smaller than the product of the  norms of the individual maps in the composition.  This is again an important factor determining power contractivity of $\bcTh$ (see also
 \cite[Experiment 3-4]{GoGaGrLaSp:21} and the discussion  in \S \ref{subsec:power_recap}). 
\end{enumerate}

\begin{table}[H]
\setlength\extrarowheight{2pt} 
\centering
\scalebox{0.9}{
\begin{tabularx}{\textwidth}{C|CCC|CCC|CCC}
\hline
\hline
${k} \backslash {\delta}$ & $h$  &  $2h$	& $4h$  &  $\frac{1}{4}$	&$\frac{1}{2}$
&  $\frac{3}{4}$	& $1-4h$ & $1-2h$ & $1-h$ 	\\ 
\hline 
$10$	&0.938	&0.889	&0.807	&0.216	&0.116	&0.102	&0.015	&0.007	&0.003\\
$20$&0.925	&0.866	&0.771	&0.237	&0.156	&0.097	&0.029	&0.014	&0.007	\\
$40$&0.906	&0.836	&0.732	&0.279	&0.180	&0.134	&0.070	&0.036	&0.018	\\
$80$&0.883	&0.804	&0.703	&0.336	&0.220	&0.124	&0.149	&0.087	&0.045\\
\hline
	\hline
\end{tabularx}}
\caption{Values of  $\| \cI^h_{+\rightarrow -} \|$  for different $\delta$ and $k$  \label{tb:imp-left2right}}
\end{table}

\begin{table}[H]
\setlength\extrarowheight{2pt} 
\centering
\scalebox{0.9}{\begin{tabularx}{\textwidth}{C|CCC|CCC|CCC}
\hline
\hline
${k} \backslash {\delta}$ & $h$  &  $2h$	& $4h$  &  $\frac{1}{4}$	&$\frac{1}{2}$   &  $\frac34$	& $1- 4h$ & $1- 2h$ & $1- h$	\\ 
\hline 
$10$	&1.000	&1.000	&1.000	&0.996	&0.980	&0.945	&0.899	&0.898	&0.897	\\
$20$&1.000	&1.000	&1.000	&1.001	&1.000	&0.999	&0.994	&0.994	&0.994	\\
$40$&1.000	&1.001	&1.001	&1.003	&1.001	&1.000	&1.000	&1.000	&1.000	\\
$80$&1.001	&1.002	&1.003	&1.002	&1.002	&1.002	&1.002	&1.000	&1.000\\
\hline
	\hline
\end{tabularx}
}

\caption{ Values of $\| \cI^h_{- \rightarrow -} \|$  for different $\delta$ and $k$ \label{tb:imp-left2left}}
\end{table}

\subsection{Tests of ORAS as a preconditioner}
Substantial numerical experiments 
on the ORAS preconditioner were already performed in \cite{GoGrSp:21,GoGaGrSp:21, GoGaGrLaSp:21}. We first summarise the results of these experiments  and then focus on some new experiments that are closely related to the theoretical results in this paper.

When the domain is decomposed into strips, the following holds.
\begin{enumerate}
\item If the norm of left-to-right impedance map becomes smaller, which is achievable by making the length of local domain and the overlap size larger, the convergence of the ORAS  iterative method becomes faster; (see \cite[Experiment 1]{GoGaGrLaSp:21}).

\item The relative error history of the ORAS iterative method has a sudden reduction of the error after each batch of $N$ steps;  (see \cite[Experiment 2]{GoGaGrLaSp:21}). Actually, the $N$th power of the ORAS error propagation matrix  is a contraction in the Helmholtz energy norm,
  while the $(N-1)$th power is not (see \cite[Table 1]{GoGaGrSp:21}). Here $N$ is the number of subdomains. 

\end{enumerate}
 
\paragraph{Third experiment.}
This experiment confirms
that the ORAS iterative method has a similar convergence property to the parallel Schwarz method, and thus is independent of  $h$  and the  polynomial degree $p$.  We decompose a rectangular domain of unit height and of length $16/3$ into eight overlapping subdomains
overlap width $1/2$. In Table \ref{tb:iterations8strips-h}-\ref{tb:iterations8strips-p}, we list the number of iterations (brackets for GMRES iterations) required to obtain a reduction of $10^{-6}$ on the  residual in Euclidean norm.
Table \ref{tb:iterations8strips-h} displays the results of experiments using different mesh sizes and using lowest order finite element space. Table \ref{tb:iterations8strips-p} displays the results of experiments using different finite-element orders and mesh size fixed at $\frac{2\pi}{10k}$. The results show
that the convergence of ORAS as an iterative method and as a preconditioner for GMRES is independent of the mesh size $h$ and the polynomial degree $p$ of the local finite element space and even improves slightly for small $h$ and large $p$. 

\begin{table}[H]
\setlength\extrarowheight{2pt} 
\centering
\scalebox{0.9}{
\begin{tabularx}{\textwidth}{C|CCCC}
\hline
\hline
$k \backslash h$ & $\frac{2\pi}{10k}$  &  $\frac{2\pi}{10k}\cdot \frac{1}{2}$	& $\frac{2\pi}{10k}\cdot \frac{1}{4}$  &  $\frac{2}{10k}\cdot \frac{1}{8}$			\\ 
\hline 
$20$	&14 (14)	&14 (13)	&14 (12)	&14 (11)\\
$40$	&14 (14)	&14 (13)	&14 (12)	&13 (11)\\
$80$	&14 (14)	&14 (13)	& 14 (11)	&12 (10)\\
$120$	&14 (14)	&14 (12)	&13 (11)	&12 (10)\\
\hline
	\hline
\end{tabularx}}
\caption{Iteration counts for the iterative method (GMRES) for different mesh size, 8 strip-type subdomains, $p=1$ 
}\label{tb:iterations8strips-h}
\end{table}

\begin{table}[H]
\setlength\extrarowheight{2pt} 
\centering
\scalebox{0.9}{
\begin{tabularx}{\textwidth}{C|CCCC}
\hline
\hline
$k \backslash p$ & $1$  &  $2$	& $3$  &  $4$			\\ 
\hline 
$20$	&14 (14)	&14 (12)	&13 (12)	&13 (11)\\
$40$	&14 (13)	&14 (13)	&13 (11)	&13 (10)\\
$80$	&14 (13)	&13 (13)	& 13 (10)	&12 (9)\\
$120$	&14 (13)	&13 (11)	&13 (9)	&12(9)\\
\hline
	\hline
\end{tabularx}}
\caption{Iteration counts for the iterative method (GMRES) for different polynomial degrees, 8 strip-type subdomains, $h=2\pi/10k$ 
}\label{tb:iterations8strips-p}
\end{table}

Finally we consider the case when a square domain is decomposed into uniform square subdomains (a so-called  checkerboard decomposition) and each subdomain is then extended to provide an overlapping cover. Previous computations showed the following.
\begin{enumerate}
\item The ORAS iterative method appears  robust as  $k$ increases provided the  subdomains are large enough and overlap is generous enough (see \cite[Experiment 6.3]{GoGrSp:21}).
\item While the boundary of the field of values of the ORAS preconditioned matrix is growing as $k$ increases,
 and the field of values contains the origin  (see \cite[Figure 2]{GoGrSp:21}), the norm of the power of the error propagation matrix is decreasing as the power  increases, and finally it becomes a contraction (see \cite[Figure 1]{GoGaGrSp:21}). 
\end{enumerate}

\paragraph{Fourth experiment.}
In this final experiment, we investigate how the mesh size and polynomial degree affect the convergence of the ORAS iterative method in the  checkerboard case. We choose a
  non-overlapping partition with diameter   $H \sim k^{-0.4}$ and extend it to overlapping subdomains by adding neighbouring elements with distance smaller than $H/4$.  Tables \ref{tb:iterations4chbrd-h}-\ref{tb:iterations4chbrd-p} list the number of iterations required to obtain $10^{-6}$ reduction on the initial residual.  These experiments indicate that when using a  checkerboard domain decomposition,
  the performance  of ORAS as an iterative solver and  preconditioner for  GMRES is also independent of $h$ and $p$. 

\begin{table}[H]
\setlength\extrarowheight{2pt} 
\centering
\scalebox{0.9}{
\begin{tabularx}{\textwidth}{C|CCCC}
\hline
\hline
$k \backslash h$ & $\frac{2\pi}{10k}$  &  $\frac{2\pi}{10k}\cdot \frac{1}{2}$	& $\frac{2\pi}{10k}\cdot \frac{1}{4}$  &  $\frac{2}{10k}\cdot \frac{1}{8}$			\\ 
\hline 
$40$	&14 (13)	&14 (13)	&14 (13)	&14 (13)\\
$80$	&18 (17)	&18 (17)	&18 (16)	&16 (15)\\
$120$	&20 (19)	&21 (19)	& 21 (18)	&18 (17)\\
$160$	&23 (22)	&23 (22)	&23 (21)	&20 (19)\\
\hline
	\hline
\end{tabularx}}
\caption{Checkerboard decomposition:~Iterations for the iterative method (GMRES) as $h \rightarrow 0$, checkerboard case, $p=1$
}\label{tb:iterations4chbrd-h}
\end{table}

\begin{table}[H]
\setlength\extrarowheight{2pt} 
\centering
\scalebox{0.9}{
\begin{tabularx}{\textwidth}{C|CCCC}
\hline
\hline
$k \backslash p$ & $1$  &  $2$	& $3$  &  $4$			\\ 
\hline 
$40$	&15 (14)	&14 (13)	&14 (13)	&13 (12)\\
$80$	&18 (17)	&18 (16)	&16 (15)	&15 (14)\\
$120$	&21 (20)	&22 (18)	& 19 (17)	&18 (16)\\
$160$	&23 (22)	&22 (21)	&20 (19)	&19 (17)\\
\hline
	\hline
\end{tabularx}}
\caption{Checkerboard decomposition:~Iterations for the iterative method (GMRES) as $p$ increases , $h=2\pi/(10k)$
}\label{tb:iterations4chbrd-p}
\end{table}

An outline approach to extend  the theory (at the PDE level)  from the strip decomposition to the checkerboard decomposition is given in \cite[\S 6.3]{GoGaGrLaSp:21}.

\bibliographystyle{plain}
\bibliography{../combined1.bib}

\end{document}